\documentclass[11pt,twoside]{amsart}
\usepackage{amsfonts,amsmath,amstext,amsbsy,amsthm,amscd,amsmath,amssymb,dsfont}
\usepackage[utf8]{inputenc}
\usepackage{hyperref}
\usepackage[T1]{fontenc}

\newcommand{\norme}[1]{\left\Vert #1\right\Vert}
\newtheorem{Lemme}{Lemma}[section]
\newtheorem{Prop}{Proposition}[section]  
\newtheorem{Def}{Definition}[section]
\newtheorem{Rmq}{Remark}[section]
\newtheorem{Thm}{Theorem}[section]

\theoremstyle{remark}

\usepackage{indentfirst}
\newcommand{\be}{\begin{equation}}
\newcommand{\ee}{\end{equation}}
\newcommand{\ba}{\begin{array}}
\newcommand{\ea}{\end{array}}
\newcommand{\bea}{\begin{eqnarray}}
\newcommand{\eea}{\end{eqnarray}}
\newcommand{\bee}{\begin{eqnarray*}}
\newcommand{\eee}{\end{eqnarray*}}

\newcommand{\B} {\mathbb{B}}
\newcommand{\C} {\mathbb{C}} 
\newcommand{\N} {\mathbb{N}}
\newcommand{\R} {\mathbb{R}}    
\renewcommand{\S} {\mathbb{S}}    
\newcommand{\Z} {\mathbb{Z}} 
\newcommand{\cB} {\mathcal{B}}    
\newcommand{\cC} {\mathcal{C}}     
   
\newcommand{\cF} {\mathcal{F}}     
\newcommand{\cH} {\mathcal{H}}     
\newcommand{\cI} {\mathcal{I}}     
     
\newcommand{\cL} {\mathcal{L}}      
\newcommand{\cM} {\mathcal{M}}  
\newcommand{\cN} {\mathcal{N}}  
\newcommand{\cO} {\mathcal{O}}

\newcommand{\cW} {\mathcal{W}}      
\newcommand{\cZ} {\mathcal{Z}}      

\def \with {\quad\!\hbox{with}\!\quad}
\def \andf {\quad\!\hbox{and}\!\quad}
\def\Id{\hbox{Id}}

\def\dZ_1{\delta\!Z_1}

\def\e{\varepsilon}
\def\d{\partial}

\def\wh{\widehat}
\def\wt{\widetilde}

\def\ddq{\dot\Delta_q}

\title[Partially dissipative hyperbolic systems]{Partially dissipative
hyperbolic systems in the critical regularity setting : the multi-dimensional case}

\author{Timothée Crin-Barat, Raphaël Danchin}

\begin{document}

\begin{abstract}
   We are concerned with quasilinear  symmetrizable partially dissipative 
   hyperbolic systems in the whole space $\R^d$ with $d\geq2.$ Following our 
   recent work \cite{CBD1} dedicated to the one-dimensional case, 
   we establish the existence of global strong  solutions  and decay estimates in the critical regularity setting
   whenever the system under consideration satisfies the so-called (SK) (for Shizuta-Kawashima)  condition.
   Our results in particular apply to  the compressible Euler system with damping in the velocity equation.
   
     Compared to the   papers by  Kawashima and Xu \cite{XK1,XK2} devoted to similar issues,
     our use of \emph{hybrid}  Besov norms with different regularity exponents in low and high 
   frequency enable us  to pinpoint optimal smallness conditions for global well-posedness
   and  to get more accurate information on the qualitative properties of the constructed solutions. 
   
   A great part of our analysis  relies on the study of a Lyapunov functional in the spirit of  that of Beauchard and Zuazua in \cite{BZ}.
   Exhibiting  a  damped mode  with faster time decay  than the whole solution also plays a key role.
 \end{abstract}

\maketitle

\section*{Introduction}
We are concerned with   first order  $n$-component systems in $\R^d$ of the type:
\begin{equation}
A^0(V)\frac{\partial V}{\partial t} + \sum_{j=1}^dA^j(V)\frac{\partial V}{\partial x_j}=H(V) \label{GEQSYM}
\end{equation}
where  the (smooth) matrices valued functions 
$A^j$  ($j=0,\cdots,d$)  and vector valued function $H$ are defined on some   open subset   $\mathcal{O}_V$ of $ \mathbb{R}^n$
and the unknown $V=V(t,x)$ depends on the time variable 
$t\in \mathbb{R}_+$ and  on the space variable $x\in\mathbb{R}^d$  ($d\geq2$).
We assume that the system is \emph{symmetrizable} and satisfies
additional structure assumptions that will be specified in the next section.

System \eqref{GEQSYM} is supplemented with  initial data $V_0\in\mathcal O_V$ at time $t=0.$
We are concerned with the existence of global strong solutions in the case where $V_0$ is close 
to some constant state $\bar V$ such that $H(\bar V)=0.$
  \medbreak
In the nondissipative case, that is if $H\equiv0,$ it is classical that   symmetrizable quasilinear 
hyperbolic systems supplemented with initial data with Sobolev regularity $H^s$ such that $s>1+d/2$  
  admit local-in-time strong solutions (see e.g. \cite{Benzoni-Serre}),
  that may develop singularities  in finite time even if the initial data are small
   (see for instance the works by Majda in \cite{Majda} or  Serre in \cite{Serre}).  
     By contrast,  if in the neighborhood of $\bar V,$ the term $H(V)$ has the `good' sign 
     and acts on each component of the solution (like e.g. $H(V)=D(V-\bar V)$ for some matrix
     $D$ having all its eigenvalues with positive real part),   then  smooth perturbations of $\bar V$
    give rise to global-in-time solutions that tend  exponentially fast to $\bar V$ when time goes to $\infty.$
   \smallbreak 
    In   most physical situations that may be  modelled by systems of the form \eqref{GEQSYM} however, 
    some components of the solution  satisfy conservation laws 
    and only \emph{partial} dissipation occurs, that is to say, the term $H(V)$ acts only 
    on a part of the solution. 
Typically, this happens in  gas dynamics  where the mass density and entropy 
are conserved,  or in numerical schemes  involving conservation laws with  relaxation.
 A well known example is the damped compressible Euler system for isentropic flows that will be addressed at the end
of the paper. For this system, it is known from the works of Wang and Tang \cite{Wang}  or  Sideris, Thomases and Wang \cite{Sideris} that the dissipative mechanism, albeit  only present in the velocity equation, can prevent the formation of singularities that would occur for $H\equiv0.$
\smallbreak
Looking for sufficient conditions on the dissipation term $H$  guaranteeing  the global existence of strong solutions for perturbations 
of a constant state $\bar V$  goes back to the thesis of Kawashima \cite{Kawa1} and to the more recent  work  by  Yong in \cite{Yong}.
Two main conditions arise. The first one is the so-called  (SK)(for \textit{Shizuta-Kawashima}) stability condition, see \cite{SK}, that ensures that the damping is strong enough to prevent the solutions emanating  from small perturbations of  $\bar V$ 
from blowing up. The second one is the existence of a (dissipative) entropy which provides a suitable symmetrisation of the system compatible with  $H$.
Thanks  to  those two conditions,  Yong \cite{Yong} obtained a global existence result  for systems that 
are more general than those that have been  considered by Kawashima. 

More recently, by taking advantage of the properties of the Green kernel of the linearized system around $\bar V$ and on the Duhamel formula,  Bianchini, Hanouzet and Natalini in \cite{BHN} pointed  out the 
convergence in $L^p$ of global solutions to $\bar V,$   with the rate ${\mathcal O}(t^{-\frac{d}{2}(1-\frac{1}{p})})$ when $t\rightarrow+\infty,$  for all $p\in[\min\{d,2\},\infty]$.  Let us further mention that 
 Kawashima and Yong proved decay estimates in regular Sobolev space in \cite{KYDecay}.
 
 A few years ago,   Kawashima and Xu in \cite{XK1} and \cite{XK2} extended  the prior works
on partially dissipative hyperbolic systems satisfying the (SK) and entropy conditions to critical   non-homogeneous Besov spaces.
 To obtain their results, they used the symmetrisation from \cite{KY} and applied a frequency localisation argument relying on the Littlewood-Paley decomposition. In their work,  the equivalence between  Condition (SK) and the existence of a compensating function allows  to exhibit   the global-in-time $L^2$ integrability properties of all the components of the solution. 
 \medbreak
  However, it is known that the condition (SK) is not optimal in the sense that there exist many systems that do not verify it but for which one can prove  global well-posedness results, see e.g. \cite{PengWangNoSk,BianchiniNataliniNoSk,HanouzetCarbou}.
 In \cite{BZ},  Beauchard and Zuazua  developed a new and systematic approach  that allows to establish global existence results and to describe large time behavior   of solutions to   partially dissipative systems that need not satisfy Condition (SK).
Looking at  the linearization of  System \eqref{GEQSYM} around a constant solution, namely
(denoting from now on $\d_t\triangleq \frac{\partial}{\partial t}$ and $\d_j\triangleq \frac{\partial}{\partial x_j}$),
\begin{equation} \label{LinearBZ}
\d_tZ+\sum_{j=1}^mA^j\d_jZ=-LZ,
\end{equation}
they show that Condition (SK) is equivalent to the Kalman maximal rank condition on the matrices $A^j$ and $L.$
More importantly, they  introduce a Lyapunov functional 
equivalent to the $L^2$ norm 
 that encodes enough information to recover dissipative properties of \eqref{LinearBZ}.  
 Considering such a functional  is motivated by the classical (linear) control theory of ODEs, 
 and is also related to  Villani's paper \cite{Villani}.
 Back to the nonlinear system \eqref{GEQSYM}, Beauchard and Zuazua  obtained  the existence
of global smooth solutions  for perturbations of a constant equilibrium $\bar V$
that satisfies (SK) Condition.  Furthermore, using arguments borrowed from  Coron’s return method \cite{Coron},
they were able to  achieve  certain cases where (SK)  does not hold.
\bigbreak
Our aim here is to extend the results we obtained recently  in the one-dimensional case \cite{CBD1} 
to multi-dimensional partially dissipative hyperbolic systems (see also the on-going work \cite{BCBT}
by the first author dedicated to  the relaxation limit of a non conservative multi-fluid system that does not satisfy the (SK) condition).
More precisely, under Condition (SK), we shall develop   Beauchard and Zuazua's approach as 
suggested  by the second author  in \cite{Handbook}  and prove the global well-posedness
of  \eqref{GEQSYM} supplemented with data that are close to $\bar V$ in 
an optimal critical regularity setting.  
As in the study of  the compressible Navier-Stokes system and related models (see e.g.   \cite{NSCLP, CMZ, NSCL2,HeDanchin})
it will appear naturally  that in order to get optimal results, one has to use functional 
 spaces with different regularity exponents in low and high  frequencies. 
Here, Beauchard and Zuazua's approach will give us the information 
that  the low frequencies (resp. high frequencies) of the solution of the linearized system behave like the heat flow (resp. are exponentially damped). Furthermore, in order to improve our low frequency analysis, we   will 
 exhibit a damped mode with better decay properties than the whole solution. Thanks to that, we will end up with  more accurate estimates and a weaker smallness condition that in prior works (in particular \cite{XK1})
and refine the decay estimates that were obtained in \cite{XK2}. 
\bigbreak

The paper is arranged as follows. In the first section, we specify the structure of the class  of partially dissipative hyperbolic systems we aim at considering, 
and explain the construction of a Lyapunov functional that will be the key to our global results. 
In Section \ref{s:main}, we  state  the main results of the paper.   
Section  \ref{s:proof}   is devoted to the proof of  a first  global existence
result and time decay estimates for general partially dissipative systems satisfying the Shizuta-Kawashima condition. 
 In section \ref{s:app},  under additional structure assumptions (that are satisfied by  
 the compressible Euler system with damping), we obtain a more accurate global 
 existence result. Some technical results are proved or recalled in Appendix.


\section{Hypotheses and method} \label{s:hyp}

In this section,  we specify our assumptions  on the system under consideration, and 
explain the main steps of our approach. 

\subsection{Friedrichs-symmetrizability}\label{ss:structure} 
 
First, we fix some constant solution $\bar V\in\cO_V$ of \eqref{GEQSYM} (thus satisfying $H(\bar V)=0$).
To ensure the local well-posedness, we  assume that   \eqref{GEQSYM} 
is  \emph{Friedrichs-symmetrizable},  namely
 that there exists a smooth function $S:V\mapsto S(V)$
defined on  $\cO_{V},$ valued in the set of symmetric and positively definite matrices such that for all $V\in\cO_{V},$ the matrices $(SA^0)(V),\cdots, (SA^d)(V)$ are symmetric and, in addition, 
$(SA^0(V))$ is definite positive. 
\medbreak
 Denoting  $\wt H\triangleq SH$ and
  $\wt A^j\triangleq SA^j$   for $j\in\{0,\cdots,\},$  System \eqref{GEQSYM} rewrites 
$$\wt A^0(V)\partial_tV + \sum_{j=1}^d\wt A^j(V)\partial_{j}V=\wt H(V).$$
Then, setting $Z\triangleq V-\bar V,$    $L\triangleq-D_V\wt H(\bar V)$ and $r(Z)\triangleq \wt H(\bar V+Z)+LZ,$  we get 
  \begin{equation}\label{eq:Z}
  \wt A^0(V)\d_tZ+\sum_{j=1}^d \wt A^j(V)\d_jZ+\mathds{L}Z=r(Z).\end{equation}
 By construction,  the remainder $r$ is at least quadratic with respect to $Z.$ 
\medbreak

Second, 
 we assume that 
System \eqref{GEQSYM} is \emph{partially dissipative} in the following meaning:
\begin{enumerate}
\item[(i)] The whole space  $\R^n$ may be decomposed into  $\mathbb{R}^n=\mathcal{M}\bigoplus\mathcal{M}^\perp$
where
$$\mathcal{M}=\bigl\{\phi\in\mathbb{R}^n, \,\: \langle\phi,\wt H(V)\rangle=0 \text{  for all } V\in\cO_{V}\}\cdotp$$ 
Hence, denoting by $\mathcal{P}$  the orthogonal projection on $\mathcal{M},$  we may write
\begin{equation}\label{eq:decompo} 
V=\left(\begin{matrix}V_1 \\ V_2 \end{matrix}\right)
\andf H(V)=\left(\begin{matrix} 0\\H_2(V) \end{matrix}\right) \end{equation} where $V_1=\mathcal{P}V\in\mathbb{R}^{n_1}$, $V_2=(I-\mathcal{P})V\in\mathbb{R}^{n_2}$ and  $n_1+n_2=n.$
\smallbreak\item[(ii)]  The linear map   $L\triangleq-D_V\wt H(\bar V)$
 is an isomorphism on $\mathcal{M}^{\perp}$  such that for some $c>0,$ 
  \begin{equation}\label{partdissip2}
\forall \eta\in \R^n,\; (L\eta|\eta)\geq c |L\eta|^2.
\end{equation}
\item[(iii)]    System \eqref{GEQSYM}  has  a \emph{block structure} that is compatible with 
decomposition \eqref{eq:decompo},  namely  all the matrices $\wt A^j$ are diagonal by blocks 
(first block being of size $n_1\times n_1$ and second one of size $n_2\times n_2$) 
and  we have $r(Z_1,0)=0$ for all $Z_1$ close to $0.$
This  entails that  $r$  is at least linear with respect to $Z_2.$ 
\end{enumerate}
  \smallbreak
According to the above assumptions and introducing the decompositions:
$$ \wt A^0=\left(\begin{matrix} \wt A^0_{1,1} & 0\\0 & \wt A^0_{2,2} \end{matrix}\right),\quad\wt A^j=\left(\begin{matrix} \wt A^j_{1,1} & \wt A^j_{1,2}\\\wt A^j_{2,1} & \wt A^j_{2,2} \end{matrix}\right),\quad
 L=\begin{pmatrix} 0\\ \mathds{L}\end{pmatrix}
\andf  r=\begin{pmatrix} 0\\ Q\end{pmatrix},$$
 System \eqref{eq:Z} may thus be rewritten   as:
\begin{equation} \left\{ \begin{matrix}\displaystyle \wt A^0_{1,1}(V)\partial_tZ_1 + \sum_{j=1}^d\left(\wt A_{1,1}^j(V)\partial_{j}Z_1+\wt A_{1,2}^j(V)\partial_{j}Z_2\right)=0,\\ \displaystyle \wt A^0_{2,2}(V)\partial_tZ_2 + \sum_{j=1}^d\left(\wt A_{2,1}^j(V)\partial_{j}Z_1+\wt A_{2,2}^j(V)\partial_{j}Z_2\right)+\mathds{L} Z_2=Q(Z). \end{matrix} \right. \label{GE}
\end{equation} 

   As we shall see in Section \ref{s:app}, the compressible Euler equations with damping, rewritten
 in suitable variables, satisfies the above assumptions about any constant state with positive density and null velocity.

\subsection{The Shizuta-Kawashima and Kalman rank conditions}

In order to specify  the  supplementary conditions on the structure 
of the system ensuring  global well-posedness and present the overall strategy, let
us consider  the linearization  of \eqref{GEQSYM} about $\bar V,$ namely:
   \begin{equation}\label{eq:Zlinear}\bar A^0\d_tZ+\sum_{j=1}^d\bar A^j \d_jZ + LZ=G \with  \bar A^j:=\wt A^j(\bar V)\ \hbox{ for }\ j=0,\cdots,d.\end{equation}
  Then,   owing to the symmetry of the matrices $\bar A^j,$  the classical energy method leads to 
  \begin{equation}\label{eq:ZZ}
   \frac 12\frac {d}{dt}\|Z\|_{L^2_{\bar A_0}}^2 + (LZ|Z) =0 \with \|Z\|_{L^2_{\bar A_0}}^2\triangleq (\bar A_0 Z|Z).\end{equation}
  On the one hand, since the matrix $\bar A_0$ is symmetric and positive definite, we have  
  \begin{equation}   \|Z\|_{L^2_{\bar A_0}}\simeq   \|Z\|_{L^2}^2.\end{equation}
 On the other hand, \eqref{partdissip2} and  the definition of $Z_2$ guarantee that
 there exists $\kappa_0>0$ such that
 \begin{equation}\label{eq:LZ}    (LZ|Z) \geq \kappa_0\|Z_2\|_{L^2}^2\quad\hbox{for all }
 \ Z\in L^2(\R^d;\R^n).\end{equation}
 Hence, \eqref{eq:ZZ} yields $L^2$-in-time integrability on the components of $Z$ experiencing 
 direct dissipation, but not on the whole solution.
 To compensate this lack of coercivity, following Beauchard and Zuazua in \cite{BZ}, we are going to introduce
  a lower order corrector $\cI$ to  track the optimal dissipation  of  the solution to \eqref{eq:Zlinear}. 
  Since it is more natural  to define that corrector  on the Fourier side,  let us  look at 
  \eqref{eq:Zlinear}  in the Fourier space, that is, denoting by $\xi\in\R^d$ the Fourier variable, 
  $$\bar A^0\d_t\wh Z+i\sum_{j=1}^d\bar A^j \xi_j\wh Z + L\wh Z=\wh G.$$
  Let us write  $\xi= \rho \omega$ with $\omega\in\S^{d-1}$ and $\rho=|\xi|.$ Then, the above system rewrites
  \begin{equation}\label{eq:Zomega}
  \d_t\wh Z +i\rho M_\omega \wh Z +N\wh Z = \bar A_0^{-1}\wh G\with 
  M_\omega\triangleq \bar A_0^{-1}\sum_{j=1}^d \omega_j\bar A^j\andf  
  N\triangleq \bar A_0^{-1} L. 
  \end{equation}
  Clearly, since $\bar A_0^{-1}$ is positive definite, \eqref{partdissip2} implies that there
  exists a positive constant (still denoted by $\kappa_0$) so that 
    \begin{equation}\label{partdissip3}
\forall \eta\in \R^n,\; (N\eta|\eta)\geq \kappa_0 |N\eta|^2.
\end{equation}
  Fix $n-1$ positive parameters $\e_1,\cdots \e_{n-1}$ (bound to be small), and set
  \begin{equation}\label{def:I}
  \cI\triangleq \Re \sum_{k=1}^{n-1} \e_k \bigl(NM_\omega^{k-1}\wh Z\cdot N M_\omega^k\wh Z\bigr)
  \end{equation}
  where $\cdot$ designates  the Hermitian scalar product in $\C^n.$ 
  \smallbreak
  For expository purpose, assume that $G\equiv0.$ Then, differentiating $\cI$ with respect to time and
   using \eqref{eq:Zomega} yields
  \begin{multline}\label{eq:14}
  \frac d{dt}\cI + \sum_{k=1}^{n-1} \e_k\rho |NM_\omega^k\wh Z|^2 
  = -\Im\sum_{k=1}^{n-1} \e_k\bigl(NM_\omega^{k-1} N\wh Z\cdot N M_\omega^k\wh Z\bigr)
  \\+\Re \sum_{k=1}^{n-1} \e_k \rho \bigl(NM_\omega^{k-1}\wh Z\cdotp NM_\omega^{k+1}\wh Z\bigr)
  -\Im\sum_{k=1}^{n-1} \e_k \bigl(NM^{k-1}_\omega\wh Z\cdotp NM_\omega^k N\wh Z\bigr)\cdotp
\end{multline}
As pointed out in \cite{BZ} (and recalled in Appendix for the reader's convenience), it is possible to choose
positive  and arbitrarily small parameters $\e_1,\cdots, \e_{n-1}$  so that \eqref{eq:14} implies for some $C>0,$
\begin{equation}\label{eq:I}
\frac d{dt} \cI +\frac12\sum_{k=1}^{n-1} \e_k\rho  |NM_\omega^k\wh Z|^2 \leq\frac {\kappa_0}{2(2\pi)^d\rho}|N\wh Z|^2
+C\e_1|N\wh Z|^2.
\end{equation}
Setting $\e_0=(2\pi)^{-d}\kappa_0/2,$ taking $\e_1$ small enough,  
integrating on $\R^d,$ using Fourier-Plancherel theorem and combining with \eqref{eq:ZZ},  we end up with 
\begin{multline}\label{eq:ZZ1}
\frac d{dt}\cL +\cH \leq 0\with \cH\triangleq\int_{\R^d} \sum_{k=0}^{n-1} \e_k \min(1,|\xi|^2) |NM_\omega^k\wh Z(\xi)|^2\,d\xi \\\andf
\cL\triangleq \|Z\|_{L^2_{\bar A_0}}^2 + \int_{\R^d} \min(|\xi|,|\xi|^{-1})\cI(\xi)\,d\xi.\end{multline}
Clearly,  if $\e_1,\cdots,\e_{n-1}$ are small enough, then $\cL\simeq \|Z\|_{L^2}^2.$ The question now is whether $\cH$ 
may be compared to   $\|Z\|_{L^2}^2.$ The answer depends on the properties of 
the support of $\wh Z_0$  and on the  possible cancellation of the following quantity:
\begin{equation}\label{eq:Nomega}
\mathcal{N}_{\bar{V}}:=\inf\biggl\{\sum_{k=0}^{n-1}\varepsilon_k|NM_\omega^kx|^2;\;x\in\S^{n-1},\,\omega\in\S^{d-1}\biggr\}\cdotp\end{equation}
At this very point, the  (SK) (for Shizuta and Kawashima) condition comes into play:
\begin{Def}
System \eqref{GEQSYM} verifies the (SK) condition  at $\bar{V}\in\mathcal{M}$ if, for all $\omega\in\S^{d-1},$
 whenever   $\phi\in\mathbb{R}^n$ satisfies   $N\phi=0$  and 
 $\lambda  \phi+M_\omega\phi=0$ for some   $\lambda\in\mathbb{R},$ we  must have  $\phi=0$. 
\end{Def}
It is clear that Condition (SK)   at $\bar V$ is equivalent to:
$$\forall \omega\in\mathbb{S}^{d-1},\:\:\: \ker N\cap\{\text{eigenvectors of  }M_\omega\}=\{0\}. $$
In order to pursue our analysis, we  need the following key result (see the proof in e.g. \cite{BZ}). 
\begin{Prop}
Let $M$ and $N$ be two matrices in  $\cM_n(\R).$
The following assertions are equivalent:
\begin{enumerate}
\item $N\phi=0$ and $\lambda \phi+ M\phi=0$ for some $\lambda\in\R$
implies $\phi=0$;
\item For every $\e_0,\dotsm,\e_{n-1}>0$, the function 
$$y\longmapsto\sqrt{\biggl(\sum_{k=0}^{n-1}\e_k|NM^ky|^2\biggr)}$$
defines a norm on $\mathbb{\R}^n$;
\end{enumerate}
\end{Prop}

Thanks to the above proposition and observing that the unit sphere $\mathbb{S}^{d-1}$ is compact, one may   conclude that 
Condition (SK)  is satisfied by the pair $(M_\omega,N)$ for all $\omega\in\S^{d-1}$ if 
and only if $\cN_{\bar{V}}>0.$ Furthermore, we note that:
\begin{itemize}
\item if $\wh Z_0$ is compactly supported then,  $\cH\gtrsim \|\nabla Z\|_{L^2}^2,$
which reveals a parabolic behavior of all components of the solution;
\item if the support of $\wh Z_0$ is away from the origin, then   $\cH\gtrsim \|Z\|_{L^2}^2,$
which corresponds to exponential decay. 
\end{itemize}
Therefore,  at the linear level,  in order to get optimal dissipative estimates, it is suitable to split the 
solution into low and high frequencies parts. 
This will actually be achieved by means of a  Littlewood-Paley decomposition
(introduced in the next section). 
Then, a great part of our analysis will consist in localizing \eqref{eq:Z} on the Fourier 
side by means of this decomposition, and to study the evolution of the
functional $\cL$  pertaining to  each part.

\subsection{The damped mode}

Another important ingredient of our analysis is the use of a `damped mode' that, somehow, 
may be seen as an eigenmode  corresponding  to the part of the solution that experiences
maximal dissipation in low frequencies. 
It is defined as follows~:
\begin{equation}\label{def:WW}
W\triangleq -\mathds{L}^{-1}\wt A^0_{2,2}(V)\d_tZ_2=Z_2+\sum_{j=1}^d \mathds{L}^{-1}\bigl(\wt A^j_{2,1}(V)\d_jZ_1+\wt A_{2,2}^j(V)\d_jZ_2\bigr)-\mathds{L}^{-1}Q(Z)\cdotp
\end{equation}
Note that
\begin{multline}\label{eq:W}\wt A^0_{2,2}(V)\d_tW+\mathds{L} W =\wt A^0_{2,2}(V)\mathds{L}^{-1}\sum_{j=1}^d\d_t\bigl(\wt A^j_{2,1}(V)\d_{j}Z_1+\wt A_{2,2}^j(V)\d_{j}Z_2\bigr)\\-\wt A^0_{2,2}(V)\mathds{L}^{-1}\partial_t Q(Z)\cdotp\end{multline}
On the left-hand side, Property \eqref{partdissip2} ensures maximal dissipation on $W.$
As  the right-hand side of \eqref{eq:W} contains only 
at least quadratic terms, or linear terms with  one derivative, it can be expected to be  negligible in low frequencies
if $Z$ is small enough. 
Furthermore, \eqref{def:WW} reveals that  $W$ is comparable to $Z_2$ in low frequencies. 
This will ensure  better integrability for  $Z_2$ than for the whole solution $Z.$


\section{Main results}\label{s:main}

Before stating our main results, introducing a few notations is in order. 

First,  we fix a homogeneous  Littlewood-Paley decomposition $(\ddq)_{q\in\Z}$
that is defined  by 
$$\ddq\triangleq\varphi(2^{-q}D)\with \varphi(\xi)\triangleq \chi(\xi/2)-\chi(\xi)$$
where $\chi$ stands for a  smooth function  with range in $[0,1],$ supported in the open ball $B(0,4/3)$ and
such that $\chi\equiv1$ on the closed ball $\bar{B}(0,3/4).$ 
We further state 
$$\dot S_q\triangleq \chi(2^{-q}D) \quad\hbox{for all }\ q\in\Z$$
and define  $\mathcal{S}'_h$ to be the set of tempered distributions $z$  such 
that $$\lim_{q\to-\infty}\|\dot S_qz\|_{L^\infty}=0.$$ 
Following  \cite{HJR}, we introduce the  homogeneous Besov semi-norms:
$$\|z\|_{\dot\B^s_{p,r}}\triangleq \bigl\| 2^{qs}\|\ddq z\|_{L^p(\R^d)}\bigr\|_{\ell^r(\Z)},$$
then  define the homogeneous Besov spaces $\dot\B^s_{p,r}$ (for any $s\in\R$ and $(p,r)\in[1,\infty]^2$)
 to be the subset of $z$ in  $\mathcal{S}'_h$ such that $\|z\|_{\dot\B^s_{p,r}}$ is finite. 
\smallbreak
Using from now on the shorthand notation 
\begin{equation}\label{eq:zq}
\ddq z\triangleq z_q,
\end{equation}
we  associate to any element $z$ of $\mathcal{S}'_h,$  its low and 
high frequency parts through 
 $$ z^{\ell}\triangleq \sum_{q\leq 0}z_q= \dot S_{1}z\andf  z^h\triangleq\sum_{q>0}z_q=({\rm Id}-\dot S_{1})z.$$
 We shall constantly use the following  Besov semi-norms for low and high frequencies: 
$$ \displaylines{\norme{z}^{\ell}_{\dot{\mathbb{B}}^{s}_{2,1}}\triangleq \sum_{q\leq 0}2^{qs}\|z_q\|_{L^2} \andf
\norme{z}^{h}_{\dot{\mathbb{B}}^{s}_{2,1}}\triangleq \sum_{q>0}2^{qs}\|z_q\|_{L^2},\cr
\norme{z}^{\ell}_{\dot{\mathbb{B}}^{s}_{2,\infty}}\triangleq \sup_{q\leq 0}2^{qs}\|z_q\|_{L^2} \andf
\norme{z}^{h}_{\dot{\mathbb{B}}^{s}_{2,\infty}}\triangleq \sup_{q>0}2^{qs}\|z_q\|_{L^2}.}
$$ 
Throughout the paper, we shall use repeatedly the following obvious fact:
\begin{equation}\label{eq:comparaison}
\|z\|^\ell_{\dot\B^{s'}_{2,r}}\leq \|z\|^\ell_{\dot\B^s_{2,r}}\andf
\|z\|^h_{\dot\B^{s'}_{2,r}}\geq \|z\|^h_{\dot\B^s_{2,r}}\ \hbox{ for }\ r=1,\infty,
\quad\hbox{whenever }\  s\leq s'.\end{equation}
\medbreak
For any Banach space $X,$ index $\rho$ in $[1,\infty]$  and time $T\in[0,\infty],$ we use the notation 
$\|z\|_{L_T^\rho(X)}\triangleq  \bigl\| \|z\|_{X}\bigr\|_{L^\rho(0,T)}.$
If $T=+\infty$, then we  just  write $\|z\|_{L^\rho(X)}.$
Finally, in the case where $z$ has $n$ components $z_j$ in $X,$ we 
  keep the notation  $\norme{z}_X$ to 
mean $\sum_{j\in\{1,\cdots,n\}} \norme{z_j}_X$. 
\medbreak
We can now state our main global existence result for System \eqref{GEQSYM}, rewritten as \eqref{eq:Z}.
\begin{Thm}\label{ThmGlobal} Let $\bar{V}$ be an equilibrium state such that $H(\bar{V})=0$
and suppose that the structure  assumptions of paragraph \ref{ss:structure} and (SK) condition are satisfied. Then, there exists a positive constant $\alpha$ 
such that for all   $Z_0\in \dot{\mathbb{B}}^{\frac{d}{2}-1}_{2,1}\cap\dot{\mathbb{B}}^{\frac{d}{2}+1}_{2,1}$  satisfying 
 \begin{equation}\label{eq:smalldata}
\cZ_0\triangleq \norme{Z_0}^\ell_{\dot{\mathbb{B}}^{\frac{d}{2}-1}_{2,1}} + \norme{Z_0}^h_{\dot{\mathbb{B}}^{\frac{d}{2}+1}_{2,1}} \leq \alpha,
\end{equation} System \eqref{eq:Z}   supplemented 
with initial data $Z_0$ admits a unique global-in-time solution $Z$ in the space E defined by
$$ Z\in \mathcal{C}_b(\mathbb{R}^+;\dot{\mathbb{B}}^{\frac{d}{2}-1}_{2,1}\cap\dot{\mathbb{B}}^{\frac{d}{2}+1}_{2,1}),\;\;\; Z^h\in L^1(\mathbb{R}^+;\dot{\mathbb{B}}^{\frac{d}{2}+1}_{2,1}), \,\;\;\; Z_1^\ell\!\in\! L^1(\mathbb{R}^+;\dot{\mathbb{B}}^{\frac{d}{2}+1}_{2,1})\!\andf\! W\!\in\! L^1(\mathbb{R}^+;\dot{\mathbb{B}}^{\frac{d}{2}-1}_{2,1}),$$
with $W$  defined according to \eqref{def:WW}. 
\medbreak
Moreover,  there exists a Lyapunov functional 
that  is equivalent to $\|Z\|_{\dot\B^{\frac{d}{2}-1}_{2,1}\cap\dot\B^{\frac{d}{2}+1}_{2,1}},$  
and  a constant $C$ depending only 
on  the matrices $A^j$ and on $H$, such that
\begin{equation} \label{eq:X0}
\cZ(t)\leq C\cZ_0\quad\hbox{for all } t\geq 0\end{equation} where
\begin{multline}\label{eq:X}
\cZ(t)\triangleq\norme{Z}^\ell_{L^\infty_t(\dot{\mathbb{B}}^{\frac{d}{2}-1}_{2,1})}+\norme{Z}^h_{L^\infty_t(\dot{\mathbb{B}}^{\frac{d}{2}+1}_{2,1})} +\norme{Z}_{L^1_t(\dot{\mathbb{B}}^{\frac{d}{2}+1}_{2,1})}\\+\norme{W}^\ell_{L^1_t(\dot{\mathbb{B}}^{\frac{d}{2}-1}_{2,1})}
 +\norme{Z_2}^\ell_{L^1_t(\dot{\mathbb{B}}^{\frac{d}{2}}_{2,1})}
 +\norme{Z_2}^\ell_{L^2_t(\dot{\mathbb{B}}^{\frac{d}{2}-1}_{2,1})}.\end{multline}
\end{Thm}

\begin{Rmq}
As is, the above theorem does not
extend to  the case $d=1.$
The reason why is that the low frequency regularity index then 
becomes negative, so that some nonlinear terms cannot be bounded in the proper spaces. 
For more details, the reader may refer to \cite{CBD1}.
\end{Rmq}
Our second result concerns the time-decay estimates of the solution we constructed in Theorem \ref{ThmGlobal}.
 \begin{Thm} \label{ThmDecay} Under the hypotheses of  Theorem \ref{ThmGlobal} 
 and if,  additionally,  $Z_0\in\dot{\mathbb{B}}^{-\sigma_1}_{2,\infty}$ for some
$\sigma_1\in\left]-\frac{d}{2},\frac{d}{2}\right]$ then, there exists 
a constant $C$ depending only on $\sigma_1$ and such that 
\begin{equation}\label{eq:Zs1}\norme{Z(t)}_{\dot{\mathbb{B}}^{-\sigma_1}_{2,\infty}}\leq C\norme{Z_0}_{\dot{\mathbb{B}}^{-\sigma_1}_{2,\infty}},\quad\forall t\geq 0. \end{equation}
Furthermore,  if $\sigma_1>1-d/2$ then, denoting 
$$\langle  t \rangle\triangleq\sqrt{1+t^2},\quad \alpha_1\triangleq \frac{\sigma_1+\frac{d}{2}-1}{2}
\andf C_0\triangleq 
\|Z_0\|^\ell_{\dot{\mathbb{B}}^{-\sigma_1}_{2,\infty}} + 
\|Z_0\|^h_{\dot{\mathbb{B}}^{\frac d2+1}_{2,1}},$$
we have the following decay estimates: 
$$
\begin{aligned}
\sup_{t\geq0} \norme{\langle t \rangle^{\frac{\sigma+\sigma_1}2}Z(t)}^{\ell}_{\dot{\mathbb{B}}^{\sigma}_{2,1}}&\leq CC_0\  \text{ if } \  -\sigma_1<\sigma\leq d/2-1,\\ 
\sup_{t\geq0}\norme{\langle t\rangle^{\frac{\sigma+\sigma_1}2+\frac12}
Z_2(t)}^\ell_{\dot{\mathbb{B}}^{\sigma}_{2,1}}&\leq CC_0 \   \text{ if } \  -\sigma_1<\sigma\leq d/2-2,\\
\sup_{t\geq0}\norme{\langle t\rangle^{\alpha_1}
Z_2(t)}^\ell_{\dot{\mathbb{B}}^{\sigma}_{2,1}}&\leq CC_0  \   \text{ if }\  \min(d/2-2,-\sigma_1)< \sigma\leq d/2-1\\ 
 \andf \sup_{t\geq0}\norme{\langle t\rangle^{2\alpha_1}Z(t)}^h_{\dot{\mathbb{B}}^{\frac{d}{2}+1}_{2,1}}&\leq C C_0. \end{aligned}$$
\end{Thm}

\begin{Rmq}
Since we have the embedding $L^1\hookrightarrow\dot{\mathbb{B}}^{-\frac{d}{2}}_{2,\infty}$, the
above statement encompasses the   classical decay assumption $Z_0\in L^1$ 
(see e.g. \cite{MatsumuraNishida} in a slightly different context). 
\end{Rmq}
\begin{Rmq} 
Owing to the presence of "direct" dissipation in the equation of $Z_2,$ the decay of the low frequencies of  $Z_2$ is stronger 
by a factor $1/2$ than the decay of  the whole solution.  
\end{Rmq}
\medbreak
If we assume in addition that:
\begin{equation}\label{StructAssum}\left\{\begin{aligned}
&\hbox{For all}\ j\in\{1,\cdots,d\},\quad A^j_{1,1}(\bar V)=0\andf D_{V_1}A^j_{1,1}(\bar V)=0;\\
 &\hbox{For all}\  j\in\{1,\cdots,d\},\quad D_{V_1}A^j_{2,1}(\bar V)=0\ (\hbox{and thus 
 also  }\ D_{V_1}A^j_{1,2}(\bar V)=0);\\
&\hbox{The function}\ r\ \hbox{is quadratic with respect to } Z_2\ 
(\hbox{i.e.}\  D^2_{V_i,V_j}r(0)=0\ \hbox{for } (i,j)\not=(2,2)),
\end{aligned} \right.\end{equation}
then   one can weaken the low frequency assumption, as we did in 
our work \cite{CBD1} dedicated to one-dimensional case,  and get: 
 \begin{Thm}\label{Thmd2}  Let the assumptions of Theorem \ref{ThmGlobal} concerning system \eqref{GEQSYM} be 
 in force and assume in addition that \eqref{StructAssum} holds true. 
 
 Then, there exists a positive constant $\alpha$ 
such that for all   $Z_0\in \dot{\mathbb{B}}^{\frac{d}{2}}_{2,1}\cap\dot{\mathbb{B}}^{\frac{d}{2}+1}_{2,1}$  satisfying 
 \begin{equation}\label{eq:smalldatabis}
\cZ'_0\triangleq  \norme{Z_0}^\ell_{\dot{\mathbb{B}}^{\frac{d}{2}}_{2,1}} + \norme{Z_0}^h_{\dot{\mathbb{B}}^{\frac{d}{2}+1}_{2,1}} \leq \alpha,
\end{equation} System \eqref{eq:Z}  supplemented with initial data $Z_0$  admits a unique global-in-time solution $Z$ in the space $F$ defined by 
$$ Z\in \mathcal{C}_b(\mathbb{R}^+;\dot{\mathbb{B}}^{\frac{d}{2}}_{2,1}\cap\dot{\mathbb{B}}^{\frac{d}{2}+1}_{2,1}),\;\;\; Z^h\in L^1(\mathbb{R}^+;\dot{\mathbb{B}}^{\frac{d}{2}+1}_{2,1}), \,\;\;\; Z_1^\ell\!\in\! L^1(\mathbb{R}^+;\dot{\mathbb{B}}^{\frac{d}{2}+2}_{2,1})\!\andf\! W\!\in\! L^1(\mathbb{R}^+;\dot{\mathbb{B}}^{\frac{d}{2}}_{2,1}).$$
Moreover,   there exists a Lyapunov functional 
that  is equivalent to $\|Z\|_{\dot\B^{\frac{d}{2}}_{2,1}\cap\dot\B^{\frac{d}{2}+1}_{2,1}},$  and we have the following a priori estimate:
\begin{multline}\label{eq:Y}\cZ'(t)\leq C\cZ'_0
\quad\hbox{where}\quad
$$\cZ'(t)\triangleq\norme{Z}^\ell_{L^\infty_t(\dot{\mathbb{B}}^{\frac{d}{2}}_{2,1})}+\norme{Z}^h_{L_t^\infty(\dot{\mathbb{B}}^{\frac{d}{2}+1}_{2,1})}\\+\norme{Z_1}^\ell_{L^1_t(\dot{\mathbb{B}}^{\frac{d}{2}+2}_{2,1})}
+\norme{Z_2}^\ell_{L^1_t(\dot{\mathbb{B}}^{\frac{d}{2}+1}_{2,1})}
+\norme{Z_2}^\ell_{L^2_t(\dot{\mathbb{B}}^{\frac{d}{2}}_{2,1})}+\norme{Z}^h_{L^1_t(\dot{\mathbb{B}}^{\frac{d}{2}+1}_{2,1})}+\norme{W}^\ell_{L^1_t(\dot{\mathbb{B}}^{\frac{d}{2}}_{2,1})}.
\end{multline}
Finally,  if,  additionally,  $Z_0\in\dot{\mathbb{B}}^{-\sigma_1}_{2,\infty}$ for some
$\sigma_1\in\left]-\frac{d}{2},\frac{d}{2}\right]$ then \eqref{eq:Zs1} is satisfied 
as well as the decay estimates that follow, \emph{up to $\sigma = d/2$} for the first
one, and with $d/2-1$ and $d/2$ instead of $d/2-2$ and $d/2-1$ for the next
two ones, with $\alpha_1$ replaced by $(\sigma_1+d/2)/2.$
\end{Thm}

\begin{Rmq}
As  will be shown in the last section, this theorem applies to the compressible Euler with damping
(see  Theorem \ref{ThmEulerd2}).
\end{Rmq}
\begin{Rmq} In contrast with  Theorem \ref{ThmGlobal},  
the  functional setting  of Theorem \ref{Thmd2} allows to obtain uniform estimates in the asymptotic $\lambda\to+\infty$ 
if  the dissipative term is $\lambda H.$  This  is  the first step 
for studying the high relaxation limit.  
\end{Rmq}


\section{Proof of Theorems \ref{ThmGlobal} and \ref{ThmDecay}} \label{s:proof}

This section is devoted to proving  the global existence of strong solutions and decay estimates
for System \eqref{GEQSYM} supplemented with initial data that are close to 
the  reference solution $\bar V,$ in the general case where the structural 
assumptions listed in Subsection \ref{ss:structure} and (SK) condition are satisfied. 

The bulk of the proof consists in establishing a priori estimates, the other steps (proving existence and uniqueness)
being more classical.  As explained before, our strategy is to first work out  
a Lyapunov functional  in Beauchard-Zuazua's style, that is equivalent to the norm that we aim
at controlling, then to combine with the study of the damped mode $W$ defined in \eqref{def:WW} 
so as to close the estimates.

\subsection{Establishing the a priori estimates}

Throughout this part,  we assume that we are given a smooth (and decaying) solution $Z$ 
of \eqref{eq:Z} on $[0,T]\times\R^d$ with $Z_0$ as initial data, satisfying
\begin{equation}\label{eq:smallZ}
\sup_{t\in[0,T]} \|Z(t)\|_{\dot\B^{\frac d2}_{2,1}}\ll1.
\end{equation} 
We shall use repeatedly that, owing to the embedding $\dot\B^{\frac d2}_{2,1}\hookrightarrow L^\infty,$ we have also
\begin{equation}\label{eq:smallZbis}
\sup_{t\in[0,T]} \|Z(t)\|_{L^\infty}\ll1.
\end{equation} 
From now on, $C>0$ designates a generic harmless constant, the 
value of which depends on the context and we denote by $(c_q)_{q\in\Z}$
nonnegative sequences such that  $\sum_{q\in\Z} c_q=1.$
\medbreak
To start with, let us rewrite \eqref{eq:Z} as  follows: 
\begin{equation} \label{controlegradient}
\bar A^0 \d_t Z+\sum_{j=1}^d\bar A^j\d_jZ+LZ= G
\end{equation} 
with $G\triangleq G_1+G_2+G_3$ and
$$\begin{aligned}G_1&\triangleq-\sum_{j=1}^d\bar A^0\left((\wt A^0(V))^{-1}\wt A^j(V)-
(\bar A^0)^{-1}\bar A^j\right)\d_jZ,\\ 
G_2&\triangleq-\bar A^0 \left((\wt A^0(V))^{-1}-(\bar A^0)^{-1}\right)LZ, \\ G_3&\triangleq \bar A^0 (\wt A^0(V))^{-1}r(Z).\end{aligned}$$
For $q\in \mathbb{Z}$, applying $\dot{\Delta}_q$  to \eqref{controlegradient} yields
\begin{equation} \label{EqFourier} 
\bar A^0\d_tZ_{q}+\sum_{j=1}^d\bar A^j\d_jZ_{q}+LZ_q=\dot{\Delta}_qG \with Z_q\triangleq \ddq Z.
\end{equation}
Our analysis will mainly consist in 
estimating for all $q\in\Z$  a functional 
 $\cL_q$ that is equivalent to the $L^2(\R^d;\R^n)$ norm  of 
$Z_q$ 
and encodes informations on the dissipative properties of the system.
That functional  will be  built from \eqref{eq:ZZ1} and, since Condition (SK) is satisfied, the number $\cN_{\bar V}$ defined in \eqref{eq:Nomega} will be  positive.
Furthermore, since the Fourier transform of $Z_q$ is localized near the frequencies of magnitude  $2^q,$ 
 the corresponding dissipation term  $\cH_q$   will satisfy
$$\cH_q\gtrsim \min (1,2^{2q}) \cL_q.$$
The prefactor $\min(1,2^{2q})$ may be seen as a gain of two derivatives 
 in low frequencies  after time integration (like  for the heat equation) whereas 
 it corresponds to exponential decay for high frequencies. 
 In our setting where the low and high frequencies of $Z_0$ belong
to the spaces $\dot\B^{\frac d2-1}_{2,1}$ and  $\dot\B^{\frac d2+1}_{2,1},$ respectively, we thus have
$$\begin{aligned}
\|Z(t)\|_{\dot\B^{\frac d2-1}_{2,1}}^\ell  + \int_0^t \|Z\|_{\dot\B^{\frac d2+1}_{2,1}}^\ell 
&\lesssim \|Z_0\|_{\dot\B^{\frac d2-1}_{2,1}}^\ell  + \int_0^t \|G \|_{\dot\B^{\frac d2-1}_{2,1}}^\ell,\\
\|Z(t)\|_{\dot\B^{\frac d2+1}_{2,1}}^h  + \int_0^t \|Z\|_{\dot\B^{\frac d2+1}_{2,1}}^h
&\lesssim \|Z_0\|_{\dot\B^{\frac d2+1}_{2,1}}^h  + \int_0^t \|G \|_{\dot\B^{\frac d2+1}_{2,1}}^h.
\end{aligned}$$
A rapid   examination reveals that   the part $G_1$ of $G$  may entail a loss of one derivative (since 
it is a combination of components of $\nabla Z$) while  $G_2$ and $G_3$ contain products of components of $Z$ and $Z_2.$ 
Overcoming the difficulty with $G_1$ will  be achieved by exploiting the symmetrizable 
character of the system under consideration and changing slightly the weight $\bar A_0$ in the definition of $\cL_q$ for the high frequencies: we shall take 
\begin{equation}\label{eq:LqHF}  \cL_q\triangleq \norme{Z_q}^2_{L^2_{\wt A_0(V)}}+2^{-q} \cI_q\quad\hbox{if }\ q\geq 0,\end{equation} 
with
\begin{equation}\label{def:Iq}
\mathcal{I}_q\triangleq\int_{\mathbb{R}^d}\sum_{k=1}^{n-1}\varepsilon_k\Re\left(( N M_\omega^{k-1}\widehat{Z_q})\cdotp( N M_\omega^k\widehat{Z_q})\right),\end{equation}
where  $\varepsilon_1,\dotsm,\varepsilon_{n-1}>0$ will be chosen small enough   (according to the Appendix).
\medbreak
For the low frequencies, we shall keep the original definition that we proposed in the analysis of 
\eqref{eq:Zlinear}, that is to say, after integrating on the whole space and using Fourier-Plancherel theorem,
\begin{equation}\label{eq:LqBF} \cL_q\triangleq \norme{Z_q}^2_{L^2_{\bar A_0}}+2^q \cI_q\quad\hbox{if }\ q<0.
\end{equation} 
However,  we will  discover that the terms $G_2$ and $G_3$ 
cannot be controlled properly  in the space  $L^1_T(\dot\B^{\frac d2-1}_{2,1})$
because $Z_2$ is, somehow, too regular !
The way to overcome the difficulty is to look for an estimate of the low frequencies of the 
damped mode $W,$ then to  compare  with $Z_2.$ 
\medbreak
We shall keep in mind all the time that if choosing the coefficients $\varepsilon_k$ small enough, then we have 
$$\sum_{k=1}^{n-1}\varepsilon_k\bigl|((M_\omega)^t)^k N^t NM_\omega^{k-1}\bigr|
\leq \frac12\frac1{(2\pi)^d},$$
whence, owing to Fourier-Plancherel theorem,  
$$|\cI_q|\leq \frac12\|Z_q\|_{L^2}.$$
Furthermore, as  $\bar A_0=A_0(\bar V)$ is definite positive and  $V\mapsto \wt A_0(V),$  continuous,
 Condition \eqref{eq:smallZbis} ensures that  $\norme{Z_q}_{L^2_{\bar A_0}}\simeq \norme{Z_q}_{L^2}$
 and   $\norme{Z_q}_{L^2_{\wt A_0(V)}}\simeq \norme{Z_q}_{L^2}.$
 Therefore, we have
\begin{equation}\label{eq:equivalence}
\cL_q  \simeq  \norme{Z_q}_{L^2}^2 
\ \hbox{for all  }\ q\in\Z.  \end{equation}


\subsubsection{Basic energy estimates}

The first step is devoted to studying the time evolution of  $\|Z_q\|_{L^2_{\wt A_0(V)}}^2$ and $\|Z_q\|_{L^2_{\bar A_0}}^2.$ The outcome 
is given in  the following proposition. 
 \begin{Prop}\label{L2estimateProp}
Let $Z$ be a smooth solution of \eqref{eq:Z} on $[0,T]\times\R^d$ satisfying \eqref{eq:smallZ}. 
 Then, for all  $s\in \left[\frac{d}{2},\frac{d}{2}+1\right]$  and $q\geq 0$, we have:
\begin{multline}
\frac12\frac{d}{dt}\norme{Z_q}^2_{L_{\wt A_0(V)}^2}+\kappa_0\norme{Z_{2,q}}^2_{L^2}\lesssim
\norme{(\nabla Z,Z_2)}_{L^\infty}\norme{Z_q}^2_{L^2}+c_q2^{-qs}\norme{\nabla Z}_{\dot{\mathbb{B}}^{\frac{d}{2}}_{2,1}}\norme{Z}_{\dot{\mathbb{B}}^s_{2,1}}\norme{Z_q}_{L^2}\\ +c_q2^{-qs}\norme{\nabla Z}_{\dot{\mathbb{B}}^{\frac d2}_{2,1}}\norme{Z_2}_{\dot{\mathbb{B}}^{s-1}_{2,1}}\norme{Z_q}_{L^2}
+ c_q2^{-qs}\bigl(\norme{Z}_{\dot{\mathbb{B}}^s_{2,1}}\norme{Z_2}_{\dot{\mathbb{B}}^{\frac d2}_{2,1}}+\norme{Z_2}_{\dot{\mathbb{B}}^s_{2,1}}\norme{Z}_{\dot{\mathbb{B}}^{\frac d2}_{2,1}}\bigr)\norme{Z_q}_{L^2}.\label{L2estimateHF}
\end{multline}
Furthermore, for all  $s'\in \left[\frac{d}{2}-1,\frac{d}{2}\right]$ and $q\leq 0,$ we have:
\begin{multline}
\frac12\frac{d}{dt}\norme{Z_q}^2_{L_{\bar A_0}^2}+\kappa_0\norme{Z_{2,q}}^2_{L^2}\lesssim\norme{\nabla Z}_{L^\infty}\norme{Z_q}^2_{L^2}+c_q2^{-qs'}\norme{\nabla Z}_{\dot{\mathbb{B}}^{\frac{d}{2}}_{2,1}}\norme{Z}_{\dot{\mathbb{B}}^{s'}_{2,1}}\norme{Z_q}_{L^2}\\+c_q2^{-qs'}\norme{Z}_{\dot{\mathbb{B}}^{\frac d2}_{2,1}}\norme{\nabla Z}_{\dot{\mathbb{B}}^{s'}_{2,1}}\norme{Z_q}_{L^2}
+ c_q2^{-qs'}\norme{Z_2}_{\dot{\mathbb{B}}^{s'}_{2,1}}\norme{Z}_{\dot{\mathbb{B}}^{\frac d2}_{2,1}}\norme{Z_q}_{L^2}.\label{L2estimateBF}\end{multline}
\end{Prop}
\begin{proof}  It relies on an energy method  implemented on \eqref{eq:Z} after 
localization in the Fourier space, and on classical commutator estimates.

In order to   prove \eqref{L2estimateHF},
apply  operator $\dot{\Delta}_q$ to \eqref{eq:Z} to get: 
$$\wt A^0(V)\d_tZ_{q}+\sum_{j=1}^d\wt A^j(V)\d_jZ_{q}+ LZ_q=R^1_q+R^2_q+\dot{\Delta}_q( r(Z))$$
with $\displaystyle R^1_q\triangleq\sum_{j=1}^d[\wt A^j(V),\dot{\Delta}_q]\d_jZ$ and $R^2_q\triangleq[\wt A^0(V),\dot{\Delta}_q]\d_tZ$.
\medbreak
Taking the $L^2(\R^d;\R^n)$ scalar product with $Z_q,$    integrating by parts in the second term 
and using the fact that $\wt A^j(V)$ is symmetric  yields
$$\displaylines{\frac12\frac d{dt}\int_{\R^d} \wt A_0(V)Z_q\cdot Z_q +\int_{\R^d} LZ_q\cdot Z_q=\frac12\int_{\R^d}\biggl(\partial_t\tilde{A}^0(V)+\sum_j\d_{j}(\wt A^j(V))\biggr)Z_q\cdot Z_q\hfill\cr\hfill
+ \int_{\R^d}(R^1_q+R^2_q)\cdot Z_q+ \int_{\R^d} \ddq r(Z)\cdot Z_q .}$$
Hence, thanks to Property \eqref{partdissip3},  we obtain
\begin{multline}\label{eq:Zq1}
\frac{1}{2}\frac{d}{dt}\norme{Z_q}^2_{L_{\wt A^0(V)}^2}+\kappa_0\norme{NZ_{q}}^2_{L^2}\leq\frac12\int_{\R^d}\biggl(\partial_t\wt{A}^0(V)
+\sum_j\d_{j}(\wt A^j(V))\biggr)Z_q\cdot Z_q\\
+ \int_{\R^d}(R^1_q+R^2_q)\cdot Z_q+ \int_{\R^d} \ddq r(Z)\cdot Z_q.\end{multline}
For the first term in the right-hand side,  we have
$$
\int_{\mathbb{R}^d}  \partial_t(\wt A^0(V))Z_q\cdotp Z_q\lesssim\norme{\d_tZ}_{L^\infty}\norme{Z_q}^2_{L^2}.$$
Hence, using the fact that  
\begin{equation}\label{eq:Zt}
\d_tZ =(\wt A^0(V))^{-1}\biggl(\wt H(\bar V+Z)-\sum_{j=1}^d \wt A^j(V)\d_j Z\biggr),\end{equation}
the smallness condition \eqref{eq:smallZbis}
and the structure of $\wt H,$  
we  get
\begin{equation}\label{eq:dta0}
\int_{\mathbb{R}^d}  \partial_t(\wt A^0(V))Z_q\cdotp Z_q\lesssim\norme{(\nabla Z,Z_2)}_{L^\infty}\norme{Z_q}^2_{L^2}.\end{equation}
For the second term in the right-hand side of \eqref{eq:Zq1}, we may write
\begin{equation}\label{eq:nablaZ}\int_{\mathbb{R}^d}\sum_{j=1}^d \partial_{j}(\wt A^j(V))Z_q\cdotp Z_q
\lesssim\norme{\nabla Z}_{L^\infty}\norme{Z_q}^2_{L^2}.\end{equation}
Bounding the commutators terms in \eqref{eq:Zq1} relies on Cauchy-Schwarz inequality and  Inequality \eqref{eq:com1} that give
$$\begin{aligned}
\int_{\mathbb{R}^d}  (R^1_q\!+\!R^2_q)\cdotp Z_q &\lesssim c_q2^{-qs}
\Bigl(\sum_j\bigl\|\nabla(\wt A^j(V))\bigr\|_{\dot{\mathbb{B}}^{\frac{d}{2}}_{2,1}}\norme{\nabla Z}_{\dot{\mathbb{B}}^{s-1}_{2,1}}
+\bigl\|\nabla(\wt A^0(V))\bigr\|_{\dot{\mathbb{B}}^{\frac{d}{2}}_{2,1}}\norme{\partial_tZ}_{\dot{\mathbb{B}}^{s-1}_{2,1}}\Bigr)
\norme{Z_q}_{L^2}\\&\lesssim c_q2^{-qs}\norme{\nabla Z}_{\dot{\mathbb{B}}^{\frac{d}{2}}_{2,1}}
\bigl(\norme{Z}_{\dot{\mathbb{B}}^s_{2,1}}+\norme{\partial_tZ}_{\dot{\mathbb{B}}^{s-1}_{2,1}}\bigr)\norme{Z_q}_{L^2}.
\end{aligned}$$
To bound $\d_tZ,$ we need  the following lemma. 
\begin{Lemme} \label{dtZ}
Under assumption \eqref{eq:smallZ}, we have   for all  $\sigma\in]-d/2, d/2]$,
$$
\norme{\d_tZ}_{\dot{\mathbb{B}}^{\sigma}_{2,1}}\lesssim 
\norme{\nabla Z}_{\dot{\mathbb{B}}^{\sigma}_{2,1}}+\min\bigl(\norme{W}_{\dot{\mathbb{B}}^{\sigma}_{2,1}},\norme{Z_2}_{\dot{\mathbb{B}}^{\sigma}_{2,1}}\bigr)\cdotp$$
\end{Lemme}
\begin{proof}
Using \eqref{eq:Zt}, Propositions \ref{LP}, \ref{Composition}  and \ref{ProprV}   yields 
$$\begin{aligned}
\norme{\d_tZ}_{\dot{\mathbb{B}}^{\sigma}_{2,1}}&\leq \biggl\|\sum_{j=1}^d \wt A^j(V)\d_jZ\biggr\|_{\dot{\mathbb{B}}^{\sigma}_{2,1}}+\norme{LZ}_{\dot{\mathbb{B}}^{\sigma}_{2,1}}+\norme{r(Z)}_{\dot{\mathbb{B}}^{\sigma}_{2,1}}
\\&\lesssim \bigl(1+\norme{Z}_{\dot{\mathbb{B}}^{\frac d2}_{2,1}}\bigr)\norme{\nabla Z}_{\dot{\mathbb{B}}^{\sigma}_{2,1}}+\norme{Z_2}_{\dot{\mathbb{B}}^{\sigma}_{2,1}}+\norme{Z}_{\dot{\mathbb{B}}^{\frac d2}_{2,1}}\norme{Z_2}_{\dot{\mathbb{B}}^{\sigma}_{2,1}}.
\end{aligned}$$
Since we assumed that $\|Z\|_{\dot\B^{\frac d2}_{2,1}}$ is small, we have
$$\norme{\d_tZ}_{\dot{\mathbb{B}}^{\sigma}_{2,1}}\lesssim 
\norme{\nabla Z}_{\dot{\mathbb{B}}^{\sigma}_{2,1}}+\norme{Z_2}_{\dot{\mathbb{B}}^{\sigma}_{2,1}}.$$
Note that, actually, $\d_tZ_1$ can be bounded by just $\nabla Z$ and that we have $\d_tZ_2=-(\wt A^0_{2,2}(V))^{-1}\mathds{L} W$  by definition of $W,$
whence the final result. 
\end{proof}
Finally, Proposition \ref{ProprV} ensures that 
$$\begin{aligned}
\int_{\R^d} \ddq  r(Z)\cdot Z_q&\lesssim c_q2^{-qs}\norme{r(Z)}_{\dot{\mathbb{B}}^s_{2,1}}\norme{Z_q}_{L^2}\\
&\lesssim c_q2^{-qs}\Bigl(\|Z_2\|_{\dot\B^{\frac d2}_{2,1}}\|Z\|_{\dot\B^s_{2,1}} + 
\|Z\|_{\dot\B^{\frac d2}_{2,1}}\|Z_2\|_{\dot\B^s_{2,1}}\Bigr)\norme{Z_q}_{L^2}\cdotp
\end{aligned}$$
Putting  all the above estimates together completes the proof of  \eqref{L2estimateHF}.
\medbreak
For proving \eqref{L2estimateBF}, since  we do not know how to control 
$\d_tZ$ in $L^1_T(\dot\B^{s'-1}_{2,1})$ for  $s'=d/2-1$ (which is the value that we will take eventually),   
we proceed slightly differently, writing the equation satisfied by $Z_q$ 
as follows: 
$$\bar A^0\d_tZ_{q}+\sum_{j=1}^d\wt A^j(V)\d_jZ_{q}+ LZ_q=R^1_q+R^3_q+\dot{\Delta}_q( r(Z))
\with R^3_q\triangleq\ddq\Bigl(\bigl(\bar A^0-\wt A^0(V)\bigr)\d_tZ\Bigl)\cdotp$$
Arguing as for proving \eqref{eq:Zq1}, we now get
\begin{multline}\label{eq:Zq2}
\frac{1}{2}\frac{d}{dt}\norme{Z_q}^2_{L_{\bar A^0}^2}+\kappa_0\norme{NZ_{q}}^2_{L^2}\leq
\frac12\int_{\R^d}\biggl(\sum_j\d_{j}(\wt A^j(V))\biggr)Z_q\cdot Z_q\\+ \int_{\R^d}(R^1_q+R^3_q)\cdot Z_q+ \int_{\R^d} \ddq r(Z)\cdot Z_q.\end{multline}
The term  $R_q^1$   may be estimated as above (with $s'$ instead of $s$), and $\ddq(r(Z))$ may be bounded
by means of \eqref{eq:R}. 
As regards   $R^3_q,$ we write that
$$\norme{R^3_q}_{L^2}\lesssim c_q2^{-qs'}\norme{\wt A^0(V)-\wt A^0(\bar{V})}_{\dot{\mathbb{B}}^{\frac d2}_{2,1}}\norme{\d_tZ}_{\dot{\mathbb{B}}^{s'}_{2,1}}.$$
Thus, using  composition, product estimates and Lemma \ref{dtZ}, we obtain
\begin{eqnarray*}
\biggl|\int_{\mathbb{R}^d} R^3_q\cdotp Z_q\biggr|&\lesssim& c_q2^{-qs'}\norme{Z}_{\dot{\mathbb{B}}^{\frac d2}_{2,1}}\norme{(\nabla Z,Z_2)}_{\dot{\mathbb{B}}^{s'}_{2,1}}\norme{Z_q}_{L^2},
\end{eqnarray*}
which leads to the desired estimate.
 \end{proof}

\subsubsection{Cross estimates}

Proposition \ref{L2estimateProp} only  allows to exhibit 
 the integrability properties of the components of $Z$ experiencing direct
 dissipation.
To recover  the dissipation for all  the components, we have to look 
at the time derivative of the quantity $\cI_q$ defined in \eqref{def:Iq}. 
To achieve it, we apply to  \eqref{EqFourier} the method that has been  explained in Section \ref{s:hyp} and leads to \eqref{eq:I}.
The only  change lies   in the (harmless) additional source term $G_q$.  In the end, integrating on $\R^d$ the obtained identity, then using the fact that ${\rm Supp}\, \wh Z_q\subset \bigl\{3\cdot 2^q/4\leq |\xi|\leq 8\cdot 2^q/3\bigr\}$ yields 
$$\frac{d}{dt}\mathcal{I}_q
+\frac{2^q}2 \sum_{k=1}^{n-1}\varepsilon_k \int_{\R^d} |NM_\omega^k \wh Z_q|^2\,d\xi
\leq \frac {2^{-q}\kappa_0}2 \|NZ_q\|_{L^2}^2 +  C\|\ddq G\|_{L^2}\|Z_q\|_{L^2}.$$
The last term may be bounded by means of Propositions \ref{LP} and \ref{Composition} (keeping
all the time in mind that \eqref{eq:smallZ} is satisfied). More precisely, 
for $G_1,$ we have for all $\sigma\in]-d/2,d/2],$ 
\begin{eqnarray}\label{eq:G1} \norme{G_1}_{\dot{\mathbb{B}}^{\sigma}_{2,1}}&\!\!\lesssim\!\!& \sum_{j=1}^d\bigl\|\bigl(\wt A^0(V)^{-1}\wt A^j(V)-(\bar A^0)^{-1}\bar A^j\bigr)\d_jZ\bigr\|_{\dot{\mathbb{B}}^{\sigma}_{2,1}}
\nonumber\\
&\!\!\lesssim\!\!& \norme{Z}_{\dot{\mathbb{B}}^{\frac{d}{2}}_{2,1}}\norme{\nabla Z}_{\dot{\mathbb{B}}^{\sigma}_{2,1}}.\end{eqnarray}
Similarly,
\begin{equation}\label{eq:G2}
\norme{G_2}_{\dot{\mathbb{B}}^{\sigma}_{2,1}}\lesssim\norme{\left((\wt A^0(V))^{-1}-(\bar A_0)^{-1}\right)LZ}_{\dot{\mathbb{B}}^{\sigma}_{2,1}}  \lesssim \norme{Z}_{\dot{\mathbb{B}}^{\frac{d}{2}}_{2,1}}\norme{Z_2}_{\dot{\mathbb{B}}^{\sigma}_{2,1}}\end{equation}
and, using Proposition \ref{ProprV}, 
\begin{equation}\label{eq:G3}
\norme{G_3}_{\dot{\mathbb{B}}^{\sigma}_{2,1}}=\norme{\bar A^0\wt A^0(V)^{-1}r(Z)}_{\dot{\mathbb{B}}^{\frac{d}{2}}_{2,1}} \lesssim \norme{Z}_{\dot{\mathbb{B}}^{\frac{d}{2}}_{2,1}}\norme{Z_2}_{\dot{\mathbb{B}}^{\sigma}_{2,1}}.\end{equation}
Hence, one can conclude that for all $\sigma\in]-d/2,d/2],$ we have
\begin{multline}\label{eq:Iq}
\frac{d}{dt}\mathcal{I}_q
+\frac{2^q}2 \sum_{k=1}^{n-1}\varepsilon_k \int_{\R^d} |NM_\omega^k \wh Z_q|^2\,d\xi\\
\leq \frac {2^{-q}\kappa_0}2 \|NZ_q\|_{L^2}^2 +  Cc_q2^{-q\sigma}\|(\nabla Z, Z_2)\|_{\dot\B^{\sigma}_{2,1}}
\|Z\|_{\dot\B^{\frac d2}_{2,1}}\|Z_q\|_{L^2}.\end{multline}

\subsubsection{Closure of the estimates : a first attempt}

Remember that since Condition (SK) is satisfied, the quantity $\cN_{\bar V}$  
defined in \eqref{eq:Nomega} is positive for any choice 
of positive parameters $\e_0,\cdots,\e_{n-1}.$ 
Consequently, if  we set 
$$\cH_q:= \frac{\kappa_0}2\|NZ_{q}\|^2 +  \min (1,2^{2q}) \sum_{k=1}^{n-1}\varepsilon_k \int_{\R^d} |NM_\omega^k \wh Z_q|^2\,d\xi$$
and use Fourier-Plancherel theorem and the equivalence \eqref{eq:equivalence}, we see that 
(up to a change of  $\kappa_0$), we have  for all $q\in\Z,$ 
\begin{equation}\label{eq:cHq}\cH_q\geq \kappa_0 \min(1,2^{2q})\cL_q.\end{equation}
Our goal is to use this inequality  to bound the quantity $\cZ$ defined in \eqref{eq:X} in terms of $\cZ_0$ only. 

Let us start with the bounds for the  low frequencies.  
Putting together Inequality \eqref{L2estimateBF}  with $s'=d/2-1$ and  the cross estimate \eqref{eq:Iq}
then, using \eqref{eq:cHq}, we get for all $q<0,$ 
$$\displaylines{
\frac d{dt} \cL_q + \kappa_0 2^{2q} \cL_q \lesssim 
\norme{\nabla Z}_{L^\infty}\norme{Z_q}^2_{L^2}+c_q2^{-q(\frac d2-1)}
\biggl(\norme{\nabla Z}_{\dot{\mathbb{B}}^{\frac{d}{2}}_{2,1}}\norme{Z}_{\dot{\mathbb{B}}^{\frac d2-1}_{2,1}}\hfill\cr\hfill +\norme{Z}_{\dot{\mathbb{B}}^{\frac d2}_{2,1}}\norme{Z_2}_{\dot{\mathbb{B}}^{\frac d2-1}_{2,1}}+\norme{Z}^2_{\dot{\mathbb{B}}^{\frac d2}_{2,1}}
+\|\nabla Z\|_{\dot\B^{\frac d2}_{2,1}}\|Z\|_{\dot\B^{\frac d2}_{2,1}}\biggr)\|Z_q\|_{L^2}.}$$
 Hence,   using \eqref{eq:equivalence}, applying Lemma \ref{SimpliCarre}, multiplying by $2^{q(\frac d2-1)},$
 using the embedding $\dot\B^{\frac d2}_{2,1}\hookrightarrow L^\infty$   and summing up on   $q<0$ gives 
$$
\displaylines{\|Z(t)\|^\ell_{\dot\cB^{\frac d2-1}_{2,1}}
  + \kappa_0\int_0^t  \|Z\|^\ell_{\dot\cB^{\frac d2+1}_{2,1}}\,d\tau
 \leq  \|Z_0\|^\ell_{\dot\cB^{\frac d2-1}_{2,1}}\hfill\cr\hfill +
 \int_0^t\Bigl(\norme{Z}_{\dot{\mathbb{B}}^{\frac{d}{2}+1}_{2,1}}\norme{Z}_{\dot{\mathbb{B}}^{\frac d2-1}_{2,1}}
 +\|Z\|_{\dot\B^{\frac d2}_{2,1}}^2 +\|Z\|_{\dot\B^{\frac d2}_{2,1}} \norme{Z_2}_{\dot{\mathbb{B}}^{\frac d2-1}_{2,1}}
 + \|Z\|_{\dot\B^{\frac d2+1}_{2,1}}\|Z\|_{\dot\B^{\frac d2}_{2,1}}\Bigr),}$$
where we used the notation
\begin{equation}\label{eq:Bsigma}
\|Z\|^\ell_{\dot\cB^{\sigma}_{2,1}}\triangleq   \sum_{q<0} 2^{q\sigma} \sqrt{\cL_q}.\end{equation}

To handle the high frequencies, we combine 
 Inequality  \eqref{L2estimateHF}  with $s=d/2+1,$  the cross estimate \eqref{eq:Iq}  and \eqref{eq:cHq}, to get for all $q\geq0,$
 \begin{multline}\label{eq:Lya3HF}\frac d{dt}\cL_q + \kappa_0 \cL_q
 \lesssim \norme{(\nabla Z,Z_2)}_{L^\infty}\norme{Z_q}^2_{L^2}\\+c_q2^{-q(\frac d2+1)}\biggl(\norme{Z}_{\dot{\mathbb{B}}^{\frac{d}{2}+1}_{2,1}}^2 +\norme{Z}_{\dot{\mathbb{B}}^{\frac d2+1}_{2,1}}\norme{Z}_{\dot{\mathbb{B}}^{\frac d2}_{2,1}}
+ \|Z_2\|_{\dot\B^{\frac d2}_{2,1}}\|Z\|_{\dot\B^{\frac d2}_{2,1}}\biggr)\norme{Z_q}_{L^2} \cdotp\end{multline}
 Hence, using the equivalence  \eqref{eq:equivalence}, Lemma \ref{SimpliCarre},  multiplying by $2^{q(\frac d2+1)}$ and summing up on 
 $q\geq0$ gives 
 $$\displaylines{\|Z(t)\|^h_{\dot\cB^{\frac d2+1}_{2,1}}
  + \kappa_0\int_0^t  \|Z\|^h_{\dot\cB^{\frac d2+1}_{2,1}}
 \leq  \|Z_0\|^h_{\dot\cB^{\frac d2+1}_{2,1}}+
\int_0^t\Bigl(\|Z\|_{\dot\B^{\frac d2+1}_{2,1}}^2 +  \|Z\|_{\dot\B^{\frac d2+1}_{2,1}}\|Z\|_{\dot\B^{\frac d2}_{2,1}}
+  \|Z_2\|_{\dot\B^{\frac d2}_{2,1}}\|Z\|_{\dot\B^{\frac d2}_{2,1}}\Bigr)}
$$
where we used the notation
$$\|Z\|^h_{\dot\cB^{\sigma}_{2,1}}\triangleq   \sum_{q\geq0} 2^{q\sigma} \sqrt{\cL_q}.$$
Let us introduce the functional   
\begin{equation}\label{def:cL}
\cL\triangleq \|Z\|^\ell_{\dot\cB^{\frac d2-1}_{2,1}}
  + \|Z\|^h_{\dot\cB^{\frac d2+1}_{2,1}}\end{equation}
   which, in light of \eqref{eq:equivalence}, is  equivalent to 
   $ \|Z\|^\ell_{\dot\B^{\frac d2-1}_{2,1}}
  + \|Z\|^h_{\dot\B^{\frac d2+1}_{2,1}},$ and thus to    $\|Z\|_{\dot\B^{\frac d2-1}_{2,1}\cap\dot\B^{\frac d2+1}_{2,1}}.$ 
  \smallbreak
  Adding up the above inequalities for the low and high frequencies, we get up to a change of $\kappa_0$ and  for all $t\in[0,T],$ 
$$\displaylines{\cL(t) + \kappa_0\int_0^t  \|Z\|_{\dot\B^{\frac d2+1}_{2,1}}
 \leq    \cL(0)+C\int_0^t\Bigl(\|Z\|_{\dot\B^{\frac d2+1}_{2,1}}^2 
\hfill\cr\hfill+  \|Z\|_{\dot\B^{\frac d2+1}_{2,1}}\|Z\|_{\dot\B^{\frac d2}_{2,1}}
 +  \|Z\|_{\dot\B^{\frac d2+1}_{2,1}}\|Z\|_{\dot\B^{\frac d2-1}_{2,1}}
+  \|Z_2\|_{\dot\B^{\frac d2-1}_{2,1}}\|Z\|_{\dot\B^{\frac d2}_{2,1}}
 +\|Z\|_{\dot\B^{\frac d2}_{2,1}}^2\Bigr)\cdotp}$$
Hence, using the  interpolation inequality
  \begin{equation}\label{eq:interpoZ}
  \|Z\|_{\dot\B^{\frac d2}_{2,1}} \lesssim   \sqrt{\|Z\|_{\dot\B^{\frac d2-1}_{2,1}}  \|Z\|_{\dot\B^{\frac d2+1}_{2,1}}}
  \lesssim\cL,\end{equation}
  and eliminating some redundant terms,   we end up with 
  \begin{equation}\label{eq:cL}
  \cL(t) +  \kappa_0\int_0^t  \|Z\|_{\dot\B^{\frac d2+1}_{2,1}}
 \leq   \cL(0)+ C\int_0^t \|Z\|_{\dot\B^{\frac d2+1}_{2,1}} \cL  
+ C\int_0^t \|Z_2\|_{\dot\B^{\frac d2-1}_{2,1}}^\ell \|Z\|_{\dot\B^{\frac d2}_{2,1}}.\end{equation}
 As we start from small data, we expect $\cL$ to be small as well so that the first term in the first integral 
 in the right-hand side may be  absorbed by the second term on the left.  
 However,  at this stage, we have no proper control on 
   $\|Z_2\|_{\dot\B^{\frac d2-1}_{2,1}}^\ell.$
Studying   the evolution of the damped mode $W,$ 
which is the aim of the next section, will enable us to overcome the difficulty.

\subsubsection{The damped mode} 
As underlined in the introduction, the function
\begin{equation*}
W\triangleq -\mathds{L}^{-1}A^0_{2,2}(V)\d_tZ_2=Z_2+\sum_{j=1}^d \mathds{L}^{-1}\bigl(\wt A^j_{2,1}(V)\d_jZ_1+\wt A_{2,2}^j(V)\d_jZ_2\bigr)-\mathds{L}^{-1}Q(Z)\cdotp
\end{equation*}
is expected  to have \emph{better} integrability properties in low frequencies than the whole solution. 
This will be a consequence of the following proposition.
\begin{Prop}\label{PropW} Let $Z$ be a smooth solution of \eqref{eq:Z} on $[0,T]\times\R^d$ satisfying \eqref{eq:smallZ},  
and  denote $\bar A^0_{2,2}\triangleq A^0_{2,2}(\bar V).$  Assume that $\sigma\in]-d/2,d/2].$ 
Then  we have for all $q<0,$  
$$\displaylines{\quad\frac12\frac d{dt}\|W_q\|_{L^2_{\bar A^0_{2,2}}}^2 + \kappa_0\|W_q\|_{L^2}^2  \lesssim
\Bigl(\|\nabla^2 Z_q\|_{L^2}+ \|\nabla W_{q}\|_{L^2}\Bigr)\|W_q\|_{L^2} \hfill\cr\hfill
+ c_q2^{-q\sigma}
 \|(W, Z_2)\|_{\dot\B^{\sigma}_{2,1}}\|Z\|_{\dot\B^{\frac d2}_{2,1}}\|W_q\|_{L^2}
+ c_q2^{-q\min(\sigma,\frac d2-1)}
 \|(\nabla Z, W)\|_{\dot\B^{\frac d2}_{2,1}}\|Z\|_{\dot\B^{\min(\sigma+1,\frac d2)}_{2,1}}\|W_q\|_{L^2}.\quad}$$
\end{Prop}
\begin{proof}
{}From  \eqref{eq:W}, we gather that 
\begin{equation}\label{eq:eqW}
\bar A^0_{2,2} \d_tW+\mathds{L} W = h\end{equation}
with $ h\triangleq h_1+ \bar A^0_{2,2} \mathds{L}^{-1}(h_2+h_3)$ and 
$$\begin{aligned}
h_1&\triangleq \bigl(\Id- \bar A^0_{2,2}(A^0_{2,2}(V))^{-1}\bigr) \mathds{L} W,\\
 h_2&\triangleq \sum_{j=1}^d\d_t\bigl(A^j_{2,1}(V)\d_jZ_1+A_{2,2}^j(V)\d_jZ_2\bigr),
 \\h_3&=-\partial_tQ(Z)\cdotp\end{aligned}$$
Applying  $\ddq$ to \eqref{eq:eqW} and taking the scalar product 
with $W_q\triangleq \ddq W$ yields, thanks to \eqref{eq:LZ}, 
\begin{equation}\label{eq:W0}\frac12\frac d{dt}\|W_q\|_{L^2_{\bar A^0_{2,2}}}^2 +\kappa_0\|W_q\|_{L^2}^2\leq  
\bigl(\|\ddq h_1\|_{L^2} + C\|\ddq h_2\|_{L^2}+ C\|\ddq h_3\|_{L^2}\bigr)\|W_q\|_{L^2}.
\end{equation}
As \eqref{eq:smallZ} is satisfied, Composition estimates readily give that for all $\sigma\in]-d/2,d/2],$ 
\begin{equation}\label{eq:h1q}
\|h_1\|_{\dot\B^{\sigma}_{2,1}} \lesssim  \|Z\|_{\dot\B^{\frac d2}_{2,1}}\|W\|_{\dot\B^{\sigma}_{2,1}}.
\end{equation}
 For bounding  $h_2,$  we use that
 for all $j\in\{1,\cdots,d\},$ 
 $$\displaylines{\d_t(A^j_{2,1}(V)\d_j Z_1 +A^j_{2,2}(V)\d_j Z_2) = D_VA^j_{2,1}(V)\d_t Z\d_jZ_1 
 +\bar A^j_{2,1}\d_t\d_jZ_1  \hfill\cr\hfill+  \bigl(A^j_{2,1}(V)-A^j_{2,1}(\bar V)\bigr)\d_t\d_jZ_1
+ D_VA^j_{2,2}(V)\d_t Z\d_jZ_2 
 +\bar A^j_{2,2}\d_t\d_jZ_2+  \bigl(A^j_{2,2}(V)-A^j_{2,2}(\bar V)\bigr)\d_t\d_jZ_2.}$$
 For $k=1,2,$ we have, according to product and composition laws, and Lemma \ref{dtZ}, 
  $$\begin{aligned} \|D_VA^j_{2,k}(V)\d_t Z\d_jZ_k\|_{\dot\B^{\sigma}_{2,1}}^\ell &\lesssim 
 \|\d_t Z\|_{\dot\B^{\frac d2}_{2,1}} \|\nabla Z\|_{\dot\B^{\sigma}_{2,1}}\\
  &\lesssim \|(\nabla  Z,W)\|_{\dot\B^{\frac d2}_{2,1}} \|Z\|_{\dot\B^{\sigma+1}_{2,1}}
  \end{aligned}$$
 as well as (provided we also have $\sigma\leq d/2-1$): 
   $$\begin{aligned} \| \bigl(A^j_{2,k}(V)-A^j_{2,k}(\bar V)\bigr)\d_t\d_jZ_k\|^\ell_{\dot\B^{\sigma}_{2,1}} &\lesssim 
 \|\d_t \nabla Z\|_{\dot\B^{\frac d2-1}_{2,1}} \|Z\|_{\dot\B^{\sigma+1}_{2,1}}\\
  &\lesssim \|(\nabla  Z,W)\|_{\dot\B^{\frac d2}_{2,1}} \|Z\|_{\dot\B^{\sigma+1}_{2,1}}.
  \end{aligned}$$
  Multiplying the first equation of \eqref{GE} (on the left) by the matrix $(\wt A^0_{1,1}(V))^{-1}$
  then differentiating with respect to $x_j,$ we discover  that 
 $\d_t\d_jZ_1$ is a combination of terms of type $A(Z)\,D^2 Z$ 
and  $B(Z)\, DZ\otimes DZ.$  Consequently, we have for all $q\leq0$   (still if  $\sigma\leq d/2-1$): 
 $$
\| \ddq(\d_t\d_jZ_1)\|_{L^2} \lesssim \|D^2Z_q\|_{L^2}
+c_q2^{-q\sigma} \|Z\|_{\dot\B^{\sigma+1}_{2,1}} \|\nabla Z\|_{\dot\B^{\frac d2}_{2,1}}.
$$
Note that if $\sigma\in]d/2-1,d/2],$ then the above inequalities  are valid (owing to \eqref{eq:comparaison})
if we change $\sigma+1$ to $d/2.$ 
\smallbreak
Finally, we have
$$
\d_t\d_jZ_2=-\d_j\bigl((\bar A^0_{2,2})^{-1}\mathds{L} W\bigr) +\d_j\bigl((\bar A^0_{2,2})^{-1}-\bigl(\wt A^0_{2,2}(V)\bigr)^{-1}\bigr)\mathds{L} W.$$
Hence, for all $q\leq0,$ and thanks to \eqref{eq:comparaison}, 
$$\begin{aligned}
\|\ddq(\d_t\d_jZ_2)\|_{L^2}&\lesssim \|\nabla W_q\|_{L^2} + c_q2^{-q\sigma}\|\bigl((\bar A^0_{2,2})^{-1}-\bigl(\wt A^0_{2,2}(V)\bigr)^{-1}\bigr)\mathds{L} W\|_{\dot\B^{\sigma}_{2,1}}\\
&\lesssim  \|\nabla W_q\|_{L^2} + c_q2^{-q\sigma}\|Z\|_{\dot\B^{\frac d2}_{2,1}}\|W\|_{\dot\B^{\sigma}_{2,1}}.\end{aligned}
$$


To bound  $h_3$, we use the fact that  $\partial_tQ(Z)=D_ZQ(Z)\partial_tZ.$ Hence, as  $Q(Z)$ is at least quadratic, we easily obtain 
from Propositions \ref{LP} and \ref{Composition} that
$$\|h_3\|_{\dot\B^{\sigma}_{2,1}} \lesssim  \|Z\|_{\dot\B^{\frac d2}_{2,1}}\|(\nabla Z,W)\|_{\dot\B^{\sigma}_{2,1}}$$
which concludes the proof.
 \end{proof}

It is now easy to obtain dissipative 
estimates for the low frequencies of $W.$
Indeed, starting from the inequality of Proposition \ref{PropW}, 
taking advantage of Lemma \ref{SimpliCarre}, multiplying the resulting inequality  with $2^{q\sigma}$
and summing up on $q<0,$ we get whenever $\sigma\in]-d/2,d/2],$ 
 \begin{multline}\label{eq:Ws}
\cW^{\sigma}(t)+ \kappa_0\int_0^t \|W\|_{\dot{\mathbb{B}}^{\frac{d}{2}}_{2,1}}^\ell \leq
\cW^{\sigma}(0)+  C\int_0^t\|(\nabla^2 Z,\nabla W)\|^\ell_{\dot\B^\sigma_{2,1}} \\
+  C\int_0^t  \|(W, Z_2)\|_{\dot\B^{\sigma}_{2,1}}\|Z\|_{\dot\B^{\frac d2}_{2,1}}
+C\int_0^t  \|(\nabla Z, W)\|_{\dot\B^{\frac d2}_{2,1}}\|Z\|_{\dot\B^{\min(\sigma+1,\frac d2)}_{2,1}}
\end{multline}
with   $\cW^\sigma\triangleq  \sum_{q<0} 2^{q\sigma} \|\ddq W\|_{L^2_{\bar A^0_{2,2}}}.$
\medbreak
Let us first apply \eqref{eq:Ws} with  $\sigma=d/2.$ 
Then we get (discarding the redundant terms): 
  \begin{multline}\label{eq:Wd2}
\cW^{\frac d2}(t)+ \kappa_0\int_0^t \|W\|_{\dot{\mathbb{B}}^{\frac{d}{2}}_{2,1}}^\ell \leq
\cW^{\frac d2}(0) \\+ C\int_0^t  \|(\nabla^2Z,\nabla W)\|_{\dot\B^{\frac d2}_{2,1}}^\ell
+ C\int_0^t  \|(\nabla Z,Z_2,W)\|_{\dot\B^{\frac d2}_{2,1}} \|Z\|_{\dot\B^{\frac d2}_{2,1}}.\end{multline}
In order to close the estimates, we also need the inequality corresponding to $\sigma=d/2-1,$ namely
  \begin{multline}\label{eq:Wd2-1}
\cW^{\frac d2-1}(t)+ \kappa_0\int_0^t \|W\|_{\dot{\mathbb{B}}^{\frac{d}{2}-1}_{2,1}}^\ell \leq
\cW^{\frac d2-1}(0)+C\int_0^t\|(\nabla Z,W)\|_{\dot{\mathbb{B}}^{\frac{d}{2}}_{2,1}}^\ell\\+ C\int_0^t\|(W,\nabla Z)\|_{\dot\B^{\frac d2}_{2,1}}\|Z\|_{\dot\B^{\frac d2}_{2,1}}+
C\int_0^t\|(Z_2,W)\|_{\dot\B^{\frac d2-1}_{2,1}}\|Z\|_{\dot\B^{\frac d2}_{2,1}}.
\end{multline}
Since
 \begin{equation}\label{eq:Z2}
 Z_2= W-\sum_{j=1}^d \mathds{L}^{-1}\bigl(A^j_{2,1}(V)\d_jZ_1+A_{2,2}^j(V)\d_jZ_2\bigr)+\mathds{L}^{-1}Q(Z)\end{equation} 
and $\|Z\|_{\dot{\mathbb{B}}^{\frac d2}_{2,1}}$ is small, we have for all $\sigma\in]-d/2,d/2],$ 
\begin{equation}\label{eq:Wsigma}
\|W-Z_2\|_{\dot\B^{\sigma}_{2,1}}\lesssim \|\nabla Z\|_{\dot\B^{\sigma}_{2,1}}+\|Z\|_{\dot{\mathbb{B}}^{\frac d2}_{2,1}}\|Z_2\|_{\dot{\mathbb{B}}^\sigma_{2,1}}.\end{equation}
Hence, $W$ may be omitted in the last term of Inequality \eqref{eq:Wd2}, and \eqref{eq:Wd2-1} becomes  
 \begin{multline}\label{eq:Wd2-1b}
\cW^{\frac d2-1}(t)+ \kappa_0\int_0^t \|W\|_{\dot{\mathbb{B}}^{\frac{d}{2}-1}_{2,1}}^\ell \leq
\cW^{\frac d2-1}(0)+C\int_0^t\|(\nabla Z,W)\|_{\dot{\mathbb{B}}^{\frac{d}{2}}_{2,1}}^\ell\\+ C\int_0^t\|\nabla Z\|_{\dot\B^{\frac d2}_{2,1}}\|Z\|_{\dot\B^{\frac d2}_{2,1}}+ C\int_0^t \|Z\|_{\dot\B^{\frac d2}_{2,1}}^2+C\int_0^t\|Z_2\|_{\dot\B^{\frac d2-1}_{2,1}}\|Z\|_{\dot\B^{\frac d2}_{2,1}}.
\end{multline}

\subsubsection{Global a priori estimates}
We are now ready  to establish the following proposition which will be the key 
to the proof of the existence part of  Theorem \ref{ThmGlobal}.  
\begin{Prop}\label{g:bound}
 Let $Z$ be a smooth solution of \eqref{eq:Z}  on $[0,T]$ satisfying the smallness condition \eqref{eq:smallZ}. 
 Then, there exist three  (small) positive parameters $\kappa_0,$ $\e$ and $\e'$ such that 
 $$ \wt\cL\triangleq \cL + \e\cW^{\frac d2}+ \e' \cW^{\frac d2-1}$$ with $\cL$  and $\cW^\sigma$
 defined in  \eqref{eq:cL} and \eqref{eq:Ws}, respectively, satisfies for all $0\leq t_0\leq t\leq T,$   
 \begin{equation}\label{eq:wtLt}\wt\cL(t)+\kappa_0\int_{t_0}^t\bigl(\|Z\|_{\dot\B^{\frac d2+1}_{2,1}}
 +\e\|W\|_{\dot\B^{\frac d2}_{2,1}}^\ell+\e'\|W\|_{\dot\B^{\frac d2-1}_{2,1}}^\ell\bigr)
 \leq \wt\cL(t_0).\end{equation}
  Furthermore, there exists a positive constant $C$ such that
\begin{equation} \label{eq:Xt} \cZ(t)\leq C \cZ_0\quad\hbox{for all }\ t\in[0,T],\end{equation}
where $\cZ_0$ and $\cZ$ have been defined in \eqref{eq:smalldata} and \eqref{eq:X}, respectively. 
\end{Prop}
\begin{proof}
{}From \eqref{eq:cL}, \eqref{eq:Wd2}, \eqref{eq:Wd2-1b} and \eqref{eq:Z2}, we  get after a few simplifications, 
$$\displaylines{
\wt\cL(t) + \kappa_0\int_0^t\Bigl(\|Z\|_{\dot\B^{\frac d2+1}_{2,1}} + \e\|W\|_{\dot\B^{\frac d2}_{2,1}}^\ell
+ \e'\|W\|_{\dot\B^{\frac d2-1}_{2,1}}^\ell\Bigr) 
\leq \wt\cL(0)
+C(\e+\e')\int_0^t\|Z\|_{\dot\B^{\frac d2+1}_{2,1}} 
\hfill\cr\hfill+C\e'\int_0^t\|W\|_{\dot\B^{\frac d2}_{2,1}}^\ell
 +C \int_0^t\|Z\|_{\dot\B^{\frac d2+1}_{2,1}}\cL
 + C\int_0^t \|Z_2\|_{\dot\B^{\frac d2-1}_{2,1}}^\ell \|Z\|_{\dot\B^{\frac d2}_{2,1}}.}$$
 Hence, choosing (positive) $\e$ and $\e'$ so that
 $$ 2C\e' \leq\kappa_0\e\andf 2C(\e+\e')\leq\kappa_0,$$
 using again \eqref{eq:Wsigma}
and the interpolation inequality  \eqref{eq:interpoZ} eventually yields:
 \begin{multline}\label{eq:cL1}
\wt\cL(t) + \kappa_0\int_0^t\Bigl(\|Z\|_{\dot\B^{\frac d2+1}_{2,1}} + \e\|W\|_{\dot\B^{\frac d2}_{2,1}}^\ell
+ \e'\|W\|_{\dot\B^{\frac d2-1}_{2,1}}^\ell\Bigr) 
\leq \wt\cL(0)\\
 +C \int_0^t\bigl(\|Z\|_{\dot\B^{\frac d2+1}_{2,1}}+\|W\|_{\dot\B^{\frac d2-1}_{2,1}}^\ell\bigr) \cL.\end{multline}
Let us denote  
$$T_0\triangleq \sup\bigl\{ t\in[0,T],\;      \sup_{\tau\in[0,t]} \wt\cL(\tau)\leq 2\wt\cL(0)\bigr\}\cdotp$$ 
Discarding the trivial case $\wt\cL(0)=0$ (corresponding to the stationary solution $\bar V$), the continuity of $\wt\cL$
ensures that $T_0>0.$ 
Now, for all $t\in[0,T_0],$ Inequality \eqref{eq:cL1} ensures that $$
\wt\cL(t) + \kappa_0\int_0^t\Bigl(\|Z\|_{\dot\B^{\frac d2+1}_{2,1}} + \e\|W\|_{\dot\B^{\frac d2}_{2,1}}^\ell
+ \e'\|W\|_{\dot\B^{\frac d2-1}_{2,1}}^\ell\Bigr) 
\leq \wt\cL(0) + 2C\wt\cL(0)\int_0^t\bigl(\|Z\|_{\dot\B^{\frac d2+1}_{2,1}}+\|W\|_{\dot\B^{\frac d2-1}_{2,1}}^\ell\bigr)\cdotp$$
Consequently, if the initial data are so small that
$4C\wt\cL(0) \leq \e'\kappa_0,$ then we deduce that
 $$\wt\cL(t) + \frac{\kappa_0}2\int_0^t\Bigl(\|Z\|_{\dot\B^{\frac d2+1}_{2,1}} + \e\|W\|_{\dot\B^{\frac d2}_{2,1}}^\ell
+ \e'\|W\|_{\dot\B^{\frac d2-1}_{2,1}}^\ell\Bigr)  \leq \wt\cL(0),$$
and thus  $T_0=T.$  Hence  \eqref{eq:wtLt} holds (with $\kappa_0/2$) on $[0,T].$
Clearly, the argument may be started from any time $t_0\in[0,T],$ which gives  \eqref{eq:wtLt} in full generality.
\medbreak
Let us finally establish \eqref{eq:Xt}. First, since $\cL$ is equivalent to $\|Z\|_{\dot\B^{\frac d2-1}_{2,1}}^\ell
+ \|Z\|_{\dot\B^{\frac d2-1}_{2,1}}^h,$ it is easy to see that, under Assumption \eqref{eq:smallZ},
we also have $\wt\cL \simeq  \|Z\|_{\dot\B^{\frac d2-1}_{2,1}}^\ell
+ \|Z\|_{\dot\B^{\frac d2+1}_{2,1}}^h.$ Combining with  \eqref{eq:wtLt}, we thus  already get
$$\norme{Z}^\ell_{L^\infty_t(\dot{\mathbb{B}}^{\frac{d}{2}-1}_{2,1})}+\norme{Z}^h_{L^\infty_t(\dot{\mathbb{B}}^{\frac{d}{2}+1}_{2,1})} +\norme{Z}_{L^1_t(\dot{\mathbb{B}}^{\frac{d}{2}+1}_{2,1})}+\norme{W}^\ell_{L^1_t(\dot{\mathbb{B}}^{\frac{d}{2}-1}_{2,1})}
\leq C\cZ_0\quad\hbox{for all }\ t\in[0,T].$$
Combining with  \eqref{eq:Wsigma},  we  discover that 
$$ \int_0^t \|Z_2\|^\ell_{\dot\B^{\frac d2}_{2,1}}
 \leq   \int_0^t \|W\|^\ell_{\dot\B^{\frac d2}_{2,1}}
+   C \int_0^t \|\nabla Z\|^\ell_{\dot\B^{\frac d2}_{2,1}}+C\int_0^t \|Z\|_{\dot{\mathbb{B}}^{\frac d2}_{2,1}}\|Z_2\|_{\dot{\mathbb{B}}^{\frac d2}_{2,1}}\lesssim \cZ_0$$
and
$$\displaylines{\|Z_2\|_{L^2_t(\dot{\mathbb{B}}^{\frac d2-1}_{2,1})}^\ell \leq \|W\|_{L^2_t(\dot{\mathbb{B}}^{\frac d2-1}_{2,1})}^\ell
+C\|\nabla Z\|_{L^2_t(\dot{\mathbb{B}}^{\frac d2-1}_{2,1})}\hfill\cr\hfill+C\|Z\|_{L^\infty_T(\dot{\mathbb{B}}^{\frac d2}_{2,1})}\|Z_2\|^h_{L^2_T(\dot{\mathbb{B}}^{\frac{d}{2}-1}_{2,1})}
+C\|Z\|_{L^\infty_T(\dot{\mathbb{B}}^{\frac d2}_{2,1})}\|Z_2\|^\ell_{L^2_T(\dot{\mathbb{B}}^{\frac{d}{2}-1}_{2,1})}.}$$
Owing to \eqref{eq:smallZ}, the last term may be absorbed by the left-hand side. Furthermore, one can 
bound the last but one thanks to \eqref{eq:comparaison} and, 
by H\"older inequality, interpolation and \eqref{eq:wtLt}, 
$$\begin{aligned}
\|\nabla Z\|_{L^2_t(\dot{\mathbb{B}}^{\frac d2-1}_{2,1})}&\lesssim
 \sqrt{ \|Z\|_{L^\infty_t(\dot{\mathbb{B}}^{\frac d2-1}_{2,1})}
  \|Z\|_{L^1_t(\dot{\mathbb{B}}^{\frac d2+1}_{2,1})}}\lesssim\cZ_0\\ 
 \|W\|_{L^2_t(\dot{\mathbb{B}}^{\frac d2-1}_{2,1})}^\ell &\leq \sqrt{ \|W\|_{L^\infty_t(\dot{\mathbb{B}}^{\frac d2-1}_{2,1})}^\ell 
  \|W\|_{L^1_t(\dot{\mathbb{B}}^{\frac d2-1}_{2,1})}^\ell} \lesssim \cZ_0, 
  \end{aligned}
  $$
  which completes the proof of the proposition. 
    \end{proof}


\subsection{Proof of  Theorem \ref{ThmGlobal}}

The starting point of the proof of existence is the following local well-posedness result 
that may be found in \cite{XK1}.
\begin{Prop} \label{ExistLocXK}
 For any data $Z_0$ in the nonhomogeneous Besov space $\mathbb{B}^{\frac{d}{2}+1}_{2,1}$, the following
 results hold true:
 \begin{enumerate}
\item Existence: there exists a positive time $T_1$, depending only  
the coefficients of the matrices $A^j,$ on $H$ and on $\norme{Z_0}_{\mathbb{B}^{\frac{d}{2}+1}_{2,1}}$
such that System \eqref{eq:Z} has a unique classical solution $Z$ with 
$$ Z \in \mathcal{C}^1([0, T_1] \times \mathbb{R}^d) \andf Z\in\mathcal{ C}([0,T_1];\mathbb{B}^{\frac{d}{2}+1}_{2,1})\cap\mathcal{ C}^1([0,T_1];\mathbb{B}^{\frac{d}{2}}_{2,1}).$$
\item Blow-up criterion:
if $T^*$ is finite, then
$$
 \int_0^{T^*}\norme{\nabla Z}_{L^\infty}dt=\infty.$$
\end{enumerate}
\end{Prop}

The proof of the existence part   of Theorem \ref{ThmGlobal} is structured as follows. First, 
we truncate the low frequencies of the data and use the above theorem to construct
a sequence $(Z^n)_{n\in\N}$ of (a priori local) approximate solutions. Then we use the previous part 
to establish that those solutions are actually global and uniformly bounded in $E.$  
In order to pass to the limit, we show that  $(Z^n)_{n\in\N}$ is a Cauchy sequence 
in $\cC([0,T];\dot\B^{\frac d2}_{2,1})$ for all $T>0.$ 
Then, we eventually check that the limit is indeed a solution of \eqref{eq:Z} and has  the required regularity.

\subsubsection*{First step. Construction of approximate solutions} 

Fix some initial data $Z_0\in \dot{\mathbb{B}}^{\frac{d}{2}-1}_{2,1}\cap\dot{\mathbb{B}}^{\frac{d}{2}+1}_{2,1}$ satisfying \eqref{eq:smalldata} and approximate it by
$$Z_{0}^n=({\rm Id}-\dot S_n)Z_0, \qquad n\geq1.$$
By construction,  $Z_{0,n}$ belongs to $\B^{\frac d2+1}_{2,1}.$ 
Consequently,  Theorem \ref{ExistLocXK} provides us with a unique maximal solution $Z^n\in \mathcal{C}([0,T_n[;{\mathbb{B}^{\frac{d}{2}+1}_{2,1}})\cap\mathcal{C}^1([0,T_n[;{\mathbb{B}^{\frac{d}{2}}_{2,1}})$. 

\subsubsection*{Second step. Uniform estimates} 
 Taking advantage of Proposition \ref{g:bound} and denoting by $\cZ^n$ the function $\cZ$ pertaining to $Z^n,$  we get  
 $\cZ^n\leq C\cZ^n_0$ as long as  $Z^n$ satisfies the smallness condition \eqref{eq:smallZ}. 
   Owing to the definition of $Z_0^n,$  we have $\cZ^n_0\leq\cZ_0$
   and we clearly have  $\|Z^n(t)\|_{\dot\B^{\frac d2}_{2,\infty}} \lesssim \cZ^n(t).$
    Hence using  a classical bootstrap argument, one can conclude that, if  $\cZ_0$ is small enough, then 
\begin{equation}\label{X22}
 \cZ^n(t)\leq  C\cZ_0, \quad\hbox{for all }\ t\in[0,T_n[.
\end{equation}
 In order to show that the solution $Z^n$ is global (that is  $T_n=+\infty$), one can 
  use the blow-up criterion 
of Theorem \ref{ExistLocXK}. However, we first have to justify that the nonhomogeneous Besov norm $\B^{\frac{d}{2}+1}_{2,1}$ of the solution is under control \emph{up to time $T_n.$}
Indeed, using the classical energy method for \eqref{eq:Z},  then the Gronwall lemma, 
we discover that   for all $t<T_n,$ 
 $$\norme{Z^n(t)}_{L^2}\leq   C\norme{Z^n_0}_{L^2} \exp\biggl(C\int_0^t\norme{\nabla Z^n}_{L^\infty}\biggr)\cdotp$$ 
 Now,  \eqref{X22} and the embedding of $\dot\B^{\frac d2}_{2,1}$ in $L^\infty$ ensure that $\nabla \cZ^n$ is in $L^1_{T_n}(L^\infty),$ from which  we deduce that 
  $Z^n$ is in  $L^\infty_{T_n}(L^2),$ and thus in 
  $L^\infty_{T_n}({\mathbb{B}}^{\frac{d}{2}+1}_{2,1})$ owing, again,  to \eqref{X22}. 
It is now easy to conclude :  we have $\nabla  Z^n$ in  $L^1_{T_n}(L^\infty),$ and 
$Z^n$  is in 
$\mathcal{ C}([0,T];\mathbb{B}^{\frac{d}{2}+1}_{2,1})\cap\mathcal{ C}^1([0,T];\mathbb{B}^{\frac{d}{2}}_{2,1})$ for all $T<T_n.$ Hence $T_n=+\infty$ and   \eqref{X22} is satisfied for all time.

\subsubsection*{Third step. Convergence} 
The following stability result will ensure 
both  the convergence of $(Z^n)_{n\in\mathbb{N}}$ and the uniqueness of our solution.
\begin{Prop} \label{PropUniq}Let $\wt{Z}=Z^1-Z^2$ where $Z^1$ and $Z^2$ are two solutions of \eqref{eq:Z}, having respectively $Z_{0}^1$ and $Z_{0}^2$ as initial data,
and belonging to the space $E$. There exists a constant $c$ such that 
if both $\|Z^1\|_{L^\infty_T(\dot\B^{\frac d2}_{2,1})}$ and $\|Z^2\|_{L^\infty_T(\dot\B^{\frac d2}_{2,1})}$
 are smaller than $c,$ then we have for all $t\in[0,T],$ 
\begin{eqnarray}\label{eq:Gronwall} \|{\wt{Z}}\|_{L^\infty_t(\dot{\mathbb{B}}^{\frac{d}{2}}_{2,1})}\lesssim \|{\wt{Z}_0}\|_{\dot{\mathbb{B}}^{\frac{d}{2}}_{2,1}}+ \int_0^t\biggl( \norme{(Z^1,Z^2)}^h_{\dot{\mathbb{B}}^{\frac{d}{2}+1}_{2,1}}+\norme{(Z^1,Z^2)}^\ell_{\dot{\mathbb{B}}^{\frac{d}{2}}_{2,1}}\biggr)\norme{\wt{Z}}_{\dot{\mathbb{B}}^{\frac{d}{2}}_{2,1}}.\end{eqnarray}
\end{Prop}
\begin{proof} Let $V^1\triangleq \bar V+Z^1$ and $V^2\triangleq \bar V+Z^2.$ Observe that 
$\wt{Z}$ is a solution of
$$\displaylines{\quad
\wt A^0(V^1)\d_t\wt{Z}+\sum_{j=1}^d\wt A^j(V^1)\d_j\wt{Z}\hfill\cr\hfill=-\wt A^0(V^1)\sum_{j=1}^d\bigl(\wt A^0(V^1)^{-1}\wt A^j(V^1)-\wt A^0(V^2)^{-1}\wt A^j(V^2)\bigr)\d_jZ^2-L\wt{Z}+r(Z^1)-r(Z^2).}$$ 
Applying $\dot{\Delta}_q$, taking the scalar product with $\wt{Z}_q$, integrating on $\mathbb{R}_+\times\mathbb{R}^d$ and using Lemma \ref{SimpliCarre}, we get for all $q\in\Z,$
$$\displaylines{\norme{\wt{Z}_q}_{L^2_{\wt A^0(V^1)}}+\kappa_0\int_0^t \norme{L\wt{Z}_q}_{L^2}\leq\norme{\wt{Z}_{0,q}}_{L^2}+\int_0^t \norme{\nabla \wt A^j(V^1)}_{L^\infty}\norme{\wt{Z}_q}_{L^2} \hfill\cr\hfill
+ \int_0^t \biggl\|\dot{\Delta}_q\sum_{j=1}^d\bigl(\wt A^0(V^1)^{-1}\wt A^j(V^1)-\wt A^0(V^2)^{-1}\wt A^j(V^2)\bigr)\d_jZ^{2}\biggr\|_{L^2}\hfill\cr\hfill+ \int_0^t \norme{\dot{\Delta}_q(r(Z^1)-r(Z^2))}_{L^2}
+\sum_j\bigl\|[\ddq,\wt A_j(V^1)]\d_j\wt Z\|_{L^2}.}
$$
Multiplying this inequality  by $2^{q\frac{d}{2}}$ and using commutator estimates, we get
\begin{multline} \label{eq:36} 2^{q\frac{d}{2}}\norme{\wt{Z}_q}_{L^2}\lesssim 2^{q\frac{d}{2}}\norme{\wt{Z}_{0,q}}_{L^2}+\int_0^t \|\nabla Z^1\|_{\dot{\mathbb{B}}^{\frac{d}{2}}_{2,1}}2^{q\frac{d}{2}}\bigr\|\wt{Z}_q\bigr\|_{L^2} \\  +\int_0^t2^{q\frac{d}{2}}  \biggl\|\dot{\Delta}_q\sum_{j=1}^d\left(\wt A^0(V^1)^{-1}\wt A^j(V^1)-\wt A^0(V^2)^{-1}\wt A^j(V^2)\right)\d_jZ^2\biggr\|_{L^2}\\ + \int_0^t  2^{q\frac{d}{2}}\norme{\dot{\Delta}_q(r(Z^1)-r(Z^2))}_{L^2}.\end{multline} 
Thanks to Propositions \ref{LP} and Inequality \eqref{eq:compo}, 
$$\norme{\left(\wt A^0(V^1)^{-1}\wt A^j(V^1)-\wt A^0(V^2)^{-1}\wt A^j(V^2)\right)\d_jZ^2}_{\dot{\mathbb{B}}^{\frac{d}{2}}_{2,1}}\lesssim 
\|\wt{Z}\bigr\|_{\dot{\mathbb{B}}^{\frac{d}{2}}_{2,1}}
\|\nabla Z^2\|_{\dot{\mathbb{B}}^{\frac{d}{2}}_{2,1}}$$
and, according to  Inequality \eqref{eq:dr}, we have 
$$  \norme{r(Z^1)-r(Z^2)}_{\dot{\mathbb{B}}^{\frac{d}{2}}_{2,1}}\lesssim  \bigl\|\wt{Z}\bigr\|_{\dot{\mathbb{B}}^{\frac{d}{2}}_{2,1}}\norme{(Z^1,Z^2)}_{\dot{\mathbb{B}}^{\frac{d}{2}}_{2,1}}.$$
Hence, summing \eqref{eq:36} on $q\in\mathbb{Z}$, we end up with 
$$\displaylines{ \norme{\wt{Z}}_{L^\infty_T(\dot{\mathbb{B}}^{\frac{d}{2}}_{2,1})}\lesssim \norme{\wt{Z}_0}_{\dot{\mathbb{B}}^{\frac{d}{2}}_{2,1}}+\int_0^t \norme{Z^1}_{\dot{\mathbb{B}}^{\frac{d}{2}+1}_{2,1}}\norme{\wt{Z}}_{\dot{\mathbb{B}}^{\frac{d}{2}}_{2,1}}\hfill\cr\hfill+\int_0^t \norme{\wt{Z}}_{\dot{\mathbb{B}}^{\frac{d}{2}}_{2,1}}\norme{\nabla Z^2}_{\dot{\mathbb{B}}^{\frac{d}{2}}_{2,1}}+\int_0^t \norme{\wt{Z}}_{\dot{\mathbb{B}}^{\frac{d}{2}}_{2,1}}\norme{(Z^1,Z^2)}_{\dot{\mathbb{B}}^{\frac{d}{2}}_{2,1}}.}$$ Splitting in low and high frequencies yields the desired estimate.
\end{proof}
The above lemma combined with the fact  that  $(Z_0^n)_{n\in\N}$ converges
to $Z_0$ in $\dot\B^{\frac d2}_{2,1}$ ensures 
 that $(Z^n)_{n\in\mathbb{N}}$ is a Cauchy sequence in $L^\infty_T(\dot{\mathbb{B}}^{\frac{d}{2}}_{2,1})$ and thus has a limit $Z$  in that space, and passing to the limit in \eqref{eq:Z} is straightforward. 
 Furthermore, using  the Fatou property of Besov spaces, we obtain that $Z^\ell\in L^\infty_T(\dot{\mathbb{B}}^{\frac{d}{2}-1}_{2,1})\cap L_T^1(\dot{\mathbb{B}}^{\frac{d}{2}+1}_{2,1}) \;\text{and}\; Z^h\in L^\infty_T(\dot{\mathbb{B}}^{\frac{d}{2}+1}_{2,1})\cap L_T^1(\dot{\mathbb{B}}^{\frac{d}{2}+1}_{2,1})$ for all $T>0,$ together with the desired bounds. 
 Time continuity of the solution may be obtained  by adapting the arguments of \cite[Chap. 4]{HJR}.

\subsubsection*{Fourth step. Uniqueness} Knowing that $Z^1$ and $Z^2$ are in $E,$ we have for all $T>0,$ 
  $$ \int_0^{T}\bigl(\|(Z^1,Z^2)\|_{\dot\B^{\frac{d}{2}}_{2,1}}^\ell
 +\|(Z^1,Z^2)\|_{\dot\B^{\frac{d}{2}+1}_{2,1}}^h\bigr)<\infty.$$
 Furthermore, one can assume with no loss of generality that $Z^1$ is the solution that we constructed before
 and thus satisfies the smallness assumption \eqref{eq:smallZ}. 
 Owing to time continuity and since $Z^2(0)=Z^1(0),$ the solution $Z^2$ also satisfies 
 \eqref{eq:smallZ} on some nontrivial time interval $[0,T],$ and combining
 Inequality \eqref{eq:Gronwall} with  Gronwall lemma allows to conclude that  $Z^1$ and $Z^2$ coincide on $[0,T].$
A  bootstrap argument then yields uniqueness on the whole half-line $\R_+.$  \qed


\subsection{Proof of Theorem \ref{ThmDecay}} \label{ss:decay}

The overall strategy is taken from the work by  Z. Xin and J. Xu in \cite{XuXin}.

\subsubsection*{First step : uniform bound in  $\dot\B^{-\sigma_1}_{2,\infty}$}

In order to establish \eqref{eq:Zs1},  one can look at System \eqref{eq:Z} as 
$$\displaylines{\bar A^0\d_tZ +\sum_{j=1}^d \bar A^j\d_jZ + LZ = f + g+ h\cr
\with  f\triangleq\sum_{j=1}^d\bigl(\bar A^j- \wt A^j(V)\bigr)\partial_{j}Z,\quad \displaystyle g\triangleq r(Z)\andf 
h\triangleq\bigl(\bar A^0-\wt A^0({V})\bigr)\partial_{t}Z,}$$ 
then   apply $\ddq$  and perform $L^2$ estimates for each $Z_q.$
\medbreak
After using Lemma \ref{SimpliCarre}, multiplying by $2^{-q\sigma_1}$ then taking the supremum on $\Z,$
 we end up (omitting the term coming from $L$ that has the `good' sign) with 
\begin{equation}\norme{Z(t)}_{\dot{\mathbb{B}}^{-\sigma_1}_{2,\infty}}\lesssim\norme{Z_0}_{\dot{\mathbb{B}}^{-\sigma_1}_{2,\infty}}+\int_0^t\norme{(f,g,h)}_{\dot{\mathbb{B}}^{-\sigma_1}_{2,\infty}}.
\end{equation}
Setting $f_j=\left( \wt A^j(V)-\bar A^j\right)\partial_{j}Z$ and  using Inequality \eqref{eq:prod3}  yields
$$\norme{f_j}_{\dot{\mathbb{B}}^{-\sigma_1}_{2,\infty}}\lesssim \norme{\wt A^j(V)- \bar A^j)}_{\dot{\mathbb{B}}^{-\sigma_1}_{2,\infty}}\norme{\partial_{j}Z}_{\dot{\mathbb{B}}^{\frac{d}{2}}_{2,1}}.$$

In order to bound $\wt A^j(V)- \bar A^j$ in $\dot\B^{-\sigma_1}_{2,\infty},$ one cannot use directly Proposition \ref{Composition} as  $-\sigma_1$ may be negative.
 However, applying Taylor formula,  product laws and  a composition estimate (see the details in the proof
 of \cite[Th. 4.1]{CBD1}), we can still obtain  if \eqref{eq:smallZ} is satisfied, 
 $$\norme{\wt A^j(V)- \bar A^j}_{\dot{\mathbb{B}}^{-\sigma_1}_{2,\infty}}\lesssim \norme{Z}_{\dot{\mathbb{B}}^{-\sigma_1}_{2,\infty}},$$
whence
$$\norme{f}_{\dot{\mathbb{B}}^{-\sigma_1}_{2,\infty}}\lesssim \norme{Z}_{\dot{\mathbb{B}}^{\frac{d}{2}+1}_{2,1}}\norme{Z}_{\dot{\mathbb{B}}^{-\sigma_1}_{2,\infty}}.
$$
For $g$, using similar arguments as in Proposition \ref{ProprV} combined with  Inequality \eqref{eq:prod3} yield
$$
\norme{g}^\ell_{\dot{\mathbb{B}}^{-\sigma_1}_{2,\infty}} \lesssim \norme{Z}_{\dot{\mathbb{B}}^{-\sigma_1}_{2,\infty}}\norme{Z_2}_{\dot{\mathbb{B}}^{\frac{d}{2}}_{2,1}}.
$$
Concerning $h$, we have, keeping Lemma \ref{dtZ} in mind, that 
$$\begin{aligned}
\norme{h}_{\dot{\mathbb{B}}^{-\sigma_1}_{2,\infty}} &\lesssim \norme{Z}_{\dot{\mathbb{B}}^{-\sigma_1}_{2,\infty}}\norme{\d_tZ}_{\dot{\mathbb{B}}^{\frac{d}{2}}_{2,1}}\\&\lesssim \norme{Z}_{\dot{\mathbb{B}}^{-\sigma_1}_{2,\infty}}\norme{(\nabla Z,Z_2)}_{\dot{\mathbb{B}}^{\frac{d}{2}}_{2,1}}.
\end{aligned}$$
Thus, regrouping all those estimates, we obtain 
\begin{equation}\norme{Z(t)}_{\dot{\mathbb{B}}^{-\sigma_1}_{2,\infty}}\lesssim\norme{Z_0}_{\dot{\mathbb{B}}^{-\sigma_1}_{2,\infty}}+\int_0^t\norme{(\nabla Z,Z_2)}_{\dot{\mathbb{B}}^{\frac{d}{2}}_{2,1}}\norme{Z}_{\dot{\mathbb{B}}^{-\sigma_1}_{2,\infty}},\qquad t\geq0.\end{equation}
Since, as pointed out before, we have
$$
\int_0^t \norme{(\nabla Z,Z_2)}_{\dot{\mathbb{B}}^{\frac{d}{2}}_{2,1}} \lesssim \cZ(t)\lesssim \cZ_0$$
and because the smallness condition \eqref{eq:smalldata} is satisfied, applying Gronwall inequality 
completes the proof of  \eqref{eq:Zs1}.
\medbreak
For the sake of completeness, one has to justify that if $Z_0$ is in $\dot B^{-\sigma_1}_{2,\infty}$ (in 
addition to \eqref{eq:smalldata}), then the solution constructed in Theorem \ref{ThmGlobal}  is 
 in $\dot B^{-\sigma_1}_{2,\infty}$ for all time. 
 This may be checked by following the construction scheme of the previous subsection.
 Indeed, recall that  the approximated solutions $Z^n$ are in $\cC^1(\R_+;\B^{\frac d2}_{2,1}).$ 
 Then, discarding the linear term $LZ^n$ (that may be handled by suitable conjugation), we get 
 $\d_t Z^n \in\cC(\R_+;L^1).$ As $L^1\hookrightarrow \dot B^{-\frac d2}_{2,\infty}$ and $\sigma_1\geq d/2,$ 
 the low frequencies  of $\d_t Z^n$ (and thus the whole $\d_tZ^n$) are  in $\cC(\R_+;\dot B^{-\sigma_1}_{2,\infty})$
 As $Z_0^n$ itself is in  $\dot B^{-\sigma_1}_{2,\infty}$  (since $\cF(Z_0^n)$ is supported away from $0$), 
 we  have $Z^n\in\cC^1(\R_+; \dot B^{-\sigma_1}_{2,\infty}).$
 Consequently, \eqref{eq:Zs1} holds for $Z^n$ and, passing to the limit, ensures that it holds for $Z,$ too.

\subsubsection*{Second step :   proof of  generic decay estimates}

According to Proposition \ref{g:bound}, the functional $\wt\cL$ introduced therein 
is nonincreasing and   equivalent to 
$\|Z\|^\ell_{\dot\B^{\frac d2-1}_{2,1}}+  \|Z\|^h_{\dot\B^{\frac d2+1}_{2,1}}.$ 
Furthermore, there exist positive $\kappa_0,$ $\e$ and $\e'$ such that 
denoting $\wt\cH\triangleq \|Z\|_{\dot\B^{\frac{d}{2}+1}_{2,1}} +\e\|W\|^\ell_{\dot\B^{\frac{d}{2}}_{2,1}}
+\e'\|W\|^\ell_{\dot\B^{\frac{d}{2}-1}_{2,1}},$ we have
$$\wt\cL(t)+\kappa_0\int_{t_0}^{t} \wt\cH\leq\wt\cL(t_0)\quad\hbox{ for all }\ 0\leq t_0\leq t.$$
Hence    and one may conclude as in \cite{CBD1} 
that $\wt\cL$ is differentiable almost everywhere  and satisfies 
\begin{equation}\label{eq:lyapunovdecay}
\frac d{dt}\wt\cL+c'\wt\cH \leq 0\quad\hbox{a. e.  on }\ \R^+.\end{equation}

Granted with this information and \eqref{eq:Zs1}, one can prove the first decay estimate of Theorem \ref{ThmDecay}
by following the general argument of  \cite{XuXin}. 
The starting point is that, provided $-\sigma_1<d/2-1,$ 
$$ \norme{Z}^\ell_{\dot{\mathbb{B}}^{\frac{d}{2}-1}_{2,1}}\lesssim \biggl(\norme{Z}^\ell_{\dot{\mathbb{B}}^{-\sigma_1}_{2,\infty}}\biggr)^{\theta_0}\biggl(\norme{Z}^\ell_{\dot{\mathbb{B}}^{\frac{d}{2}+1}_{2,1}}\biggr)^{(1-\theta_0)} \with \theta_0=\frac{2}{d/2+1+\sigma_1}\cdotp$$
Inequality  \eqref{eq:Zs1} thus implies that  
$$\norme{Z}^\ell_{\dot{\mathbb{B}}^{\frac{d}{2}+1}_{2,1}}\gtrsim  
 \bigl(\norme{Z}^\ell_{\dot{\mathbb{B}}^{\frac{d}{2}-1}_{2,1}}\bigr)^{\frac1{1-\theta_0}} \|Z_0\|_{\dot\B^{-\sigma_1}_{2,\infty}}^{-\frac{\theta_0}{1-\theta_0}}.$$
For  the high frequencies term, using the estimate of Theorem \ref{ThmGlobal}, one can just write:
$$\norme{Z}^h_{\dot{\mathbb{B}}^{\frac{d}{2}+1}_{2,1}}\gtrsim
 \biggl(\norme{Z}^h_{\dot{\mathbb{B}}^{\frac{d}{2}+1}_{2,1}}\biggr)^{\frac{1}{1-\theta_0}}
 \|Z_0\|_{\dot\B^{\frac d2-1}_{2,1}\cap\dot\B^{\frac d2+1}_{2,1}}^{-\frac{\theta_0}{1-\theta_0}}.$$
 Hence, there exists a (small) constant $c$ such that 
 $$\frac{d}{dt} \wt\cL +cC_0^{-\frac{\theta_0}{1-\theta_0}}
 \wt\cL^{\frac{1}{1-\theta_0}}\leq 0\with C_0\triangleq \|Z_0\|_{\dot\B^{-\sigma_1}_{2,\infty}\cap \dot\B^{\frac d2+1}_{2,1}}.$$
Integrating, this gives us
$$\wt\cL(t)\leq \biggl(1+c\,\frac{\theta_0}{1-\theta_0}\biggl(\frac{\wt\cL(0)}{C_0}\biggr)^{\frac{\theta_0}{1-\theta_0}}t\biggr)^{1-\frac1{\theta_0}}\wt\cL(0)$$
whence, since $\wt\cL \leq \wt\cL(0)\lesssim \cZ_0 \lesssim C_0,$
\begin{equation}\norme{Z(t)}^\ell_{\dot{\mathbb{B}}^{\frac{d}{2}-1}_{2,1}}+\norme{Z(t)}^h_{\dot{\mathbb{B}}^{\frac{d}{2}+1}_{2,1}}\lesssim(1+t)^{-\alpha_1}\cZ_0 \with \alpha_1=\frac{d/2-1+\sigma_1}{2}\cdotp\label{Decayalpha}
\end{equation}
 The decay rates in $\dot{\mathbb{B}}^{\sigma}_{2,1}$ for all $\sigma\in ]-\sigma_1,d/2-1]$  follow from 
  Inequalities \eqref{eq:Zs1} and  \eqref{Decayalpha},  and interpolation inequalities.

\subsubsection*{Third step:  decay enhancement for the damped mode}

{}From \eqref{eq:W0} and Lemma \ref{SimpliCarre}, 
  one can  get for all $\sigma\in]-\sigma_1,d/2-1],$
$$ \cW^\sigma(t)  \leq e^{-ct}\cW^\sigma(0) +C\int_0^te^{-c(t-\tau)}\norme{h(\tau)}^\ell_{\dot{\mathbb{B}}^{\sigma}_{2,1}} d\tau.
$$
Hence, since $\cW^\sigma\approx \|W\|^\ell_{\dot{\mathbb{B}}^{\sigma}_{2,1}}$ and using 
the estimates of $h$ pointed out in the proof of Proposition \ref{PropW},  we get
$$\displaylines{
 \|W(t)\|^\ell_{\dot{\mathbb{B}}^{\sigma}_{2,1}}\lesssim e^{-t}   \|W_0\|^\ell_{\dot{\mathbb{B}}^{\sigma}_{2,1}}
 \hfill\cr\hfill+\int_0^te^{-(t-\tau)}\Bigl(\|Z\|_{\dot\B^{\frac d2}_{2,1}}\|(Z_2,W)\|_{\dot\B^\sigma_{2,1}}
 +\|(\nabla Z,W)\|_{\dot\B^{\frac d2}_{2,1}}\|Z\|_{\dot\B^{\sigma+1}_{2,1}} 
 +\|(\nabla^2Z,\nabla W)\|_{\dot\B^\sigma_{2,1}}^\ell \Bigr),}$$
 whence 
 $$\displaylines{
 \|W(t)\|^\ell_{\dot{\mathbb{B}}^{\sigma}_{2,1}}\lesssim e^{-t}   \|W_0\|^\ell_{\dot{\mathbb{B}}^{\sigma}_{2,1}}
 \hfill\cr\hfill+\int_0^te^{-(t-\tau)}\Bigl(\|(\nabla Z,Z_2)\|_{\dot\B^\sigma_{2,1}}\|Z\|_{\dot\B^{\frac d2}_{2,1}}
 +\|(\nabla Z,Z_2)\|_{\dot\B^{\frac d2}_{2,1}}\|Z\|_{\dot\B^{\sigma+1}_{2,1}} 
 +\|\nabla Z\|_{\dot\B^\sigma_{2,1}}^\ell \Bigr)\cdotp}$$
 In light of  the previous step,  the worst decay comes from the last term.
 In order to be allowed to use the corresponding estimate however, we need $\sigma +1 \leq d/2-1.$
 If that condition is satisfied then, setting $\beta=(\sigma+\sigma_1+1)/2,$   the above inequality implies that 
 $$ \langle t\rangle^\beta  \|W(t)\|^\ell_{\dot{\mathbb{B}}^{\sigma}_{2,1}}\lesssim   \langle t\rangle^\beta    e^{-ct}   
 \|W_0\|^\ell_{\dot{\mathbb{B}}^{\sigma}_{2,1}} + C_0\int_0^t\frac{\langle t\rangle^\beta }
 { \langle \tau\rangle^\beta }\, e^{-c(t-\tau)}\,d\tau\lesssim C_0.$$      
 In the case $\sigma+1> d/2-1,$ one can  use the fact that $\|W(t)\|^\ell_{\dot{\mathbb{B}}^{\sigma}_{2,1}}\lesssim 
  \|W(t)\|^\ell_{\dot{\mathbb{B}}^{\frac d2-2}_{2,1}},$ and  the above argument thus just implies that 
 $$ \norme{W(t)}^\ell_{\dot{\mathbb{B}}^{\sigma}_{2,1}}\lesssim (1+t)^{-\alpha_1}.$$
Keeping in mind \eqref{eq:Wsigma}, one can  conclude that $ \norme{Z_2}^\ell_{\dot{\mathbb{B}}^{\sigma}_{2,1}}$ satisfies the 
same decay estimates as $W.$

\subsubsection*{Last step : high frequencies decay}

Let us start from \eqref{eq:Lya3HF}.  The usual method based on  Lemma \ref{SimpliCarre} leads
after multiplying   by $\langle t \rangle^{2\alpha_1}$ (where $\alpha_1$ comes from \eqref{Decayalpha})
 yields
\begin{multline}
\norme{\langle t\rangle^{2\alpha_1}Z(t)}^h_{\dot{\mathbb{B}}^{\frac{d}{2}+1}_{2,1}}\leq e^{-ct}\norme{Z_0}^h_{\dot{\mathbb{B}}^{\frac{d}{2}+1}_{2,1}}+\int_0^t\langle t\rangle^{2\alpha_1}e^{-c(t-\tau)}\norme{Z}_{\dot{\mathbb{B}}^{\frac{d}{2}+1}_{2,1}} \biggl(\norme{Z}^h_{\dot{\mathbb{B}}^{\frac{d}{2}+1}_{2,1}}+\norme{Z}^\ell_{\dot{\mathbb{B}}^{\frac{d}{2}}_{2,1}}\biggr)d\tau \\+\int_0^t\langle t\rangle^{2\alpha_1}e^{-c(t-\tau)}\norme{Z}^\ell_{\dot{\mathbb{B}}^{\frac{d}{2}}_{2,1}}\norme{Z_2}^\ell_{\dot{\mathbb{B}}^{\frac{d}{2}}_{2,1}}\,d\tau. \label{DecayHfZ}
\end{multline}
Thanks to \eqref{Decayalpha}, the first quadratic  term may be bounded as follows:
$$\int_0^t\langle t\rangle^{2\alpha_1}e^{-c(t-\tau)}\norme{Z}_{\dot{\mathbb{B}}^{\frac{d}{2}+1}_{2,1}} \norme{Z}^h_{\dot{\mathbb{B}}^{\frac{d}{2}+1}_{2,1}}\leq\int_0^t\biggl(\frac{\langle t \rangle}{\langle\tau\rangle}\biggr)^{2\alpha_1} e^{-c(t-\tau)} \bigl(\langle\tau\rangle^{\alpha_1}\norme{Z}_{\dot{\mathbb{B}}^{\frac{d}{2}+1}_{2,1}}\bigr)^2d\tau
\lesssim C_0,$$
and the other  terms of the right-hand side of \eqref{DecayHfZ} may be bounded similarly.  
This completes the proof of Theorem \ref{ThmDecay}.

\section{The proof of Theorem \ref{Thmd2} and application to the
compressible Euler system}\label{s:app}

This section is devoted to the proof of Theorem \ref{Thmd2}, that is to say 
to a refinement of Theorem \ref{ThmGlobal} 
corresponding to the case where  System \eqref{GEQSYM} satisfies the extra conditions listed in \eqref{StructAssum}.
As an application, we shall obtain a global existence statement for  the compressible Euler with damping,
in a new functional framework, and will specify the dependency of the estimates 
with respect to the relaxation (or damping) parameter. 

\subsection{Proof of Theorem \ref{Thmd2}}

Proving existence and uniqueness being very similar to what we did before, 
we focus on establishing a priori estimates for a smooth solution $Z$
of \eqref{eq:Z} on $[0,T]\times\R^d,$ satisfying the smallness condition \eqref{eq:smallZ}. 
The general strategy is the same as in the previous section, and we shall mainly underline the places where having the structure
\eqref{StructAssum} comes into play. 

The first difference is in the following refinement of 
Lemma \ref{dtZ}
\begin{Lemme} \label{dtZbis}
Under hypotheses \eqref{StructAssum} and \eqref{eq:smallZ}, we have   for all  $\sigma\in]-d/2, d/2]$,
$$\begin{aligned}
\norme{\d_tZ_1}_{\dot{\mathbb{B}}^{\sigma}_{2,1}}&\lesssim 
\norme{\nabla Z_2}_{\dot{\mathbb{B}}^{\sigma}_{2,1}}+ \|Z_2\|_{\dot\B^{\frac d2}_{2,1}}
\|\nabla Z_1\|_{\dot\B^{\sigma}_{2,1}},\\
\norme{\d_tZ_2}_{\dot{\mathbb{B}}^{\sigma}_{2,1}}&\lesssim 
\norme{W}_{\dot{\mathbb{B}}^{\sigma}_{2,1}}.\end{aligned}$$
\end{Lemme}
\begin{proof}
The second inequality has been proved before (see Lemma \ref{dtZ}). 
The first one relies on  the decomposition 
\begin{equation}\label{eq:dtZ1}
\partial_tZ_1=- \sum_{j=1}^d
(\wt A^0_{1,1}(V))^{-1}\left(\wt A_{1,1}^j(V)\partial_{j}Z_1+\wt A_{1,2}^j(V)\partial_{j}Z_2\right)\cdotp\end{equation}
As the function $V\mapsto \wt A^0_{1,1}(V))^{-1}\wt A_{1,1}^j(V)$ vanishes at $\bar V$
and is linear with respect to $Z_2,$  Propositions \ref{LP}, \ref{Composition} and Condition \eqref{eq:smallZ}
guarantee the  desired  inequality.

\end{proof}

\subsubsection{Basic energy estimates}
As for Theorem \ref{ThmGlobal},  the first step consists in proving estimates 
  for $\|Z_q\|_{L^2_{\wt A_0(V)}}^2$ and $\|Z_q\|_{L^2_{\bar A_0}}^2.$ 
   \begin{Prop}\label{L2estimatePropBFd2}
Let $Z$ be a smooth solution \eqref{eq:Z} on $[0,T]$ satisfying \eqref{eq:smallZ}. 
 Then, under Condition \eqref{StructAssum}, we have
 for all  $q\geq 0$,
\begin{equation}
\frac12\frac{d}{dt}\norme{Z_q}^2_{L_{\wt A_0(V)}^2}+\kappa_0\norme{Z_{2,q}}^2_{L^2}\lesssim
c_q2^{-q(\frac d2+1)}\norme{(W,\nabla Z)}_{\dot{\mathbb{B}}^{\frac d2}_{2,1}}\norme{Z}_{\dot{\mathbb{B}}^{\frac d2+1}_{2,1}}\norme{Z_q}_{L^2},\label{eq:L2estimateHF}
\end{equation}
and  for all  $q\leq 0,$ 
\begin{multline}
\frac12\frac{d}{dt}\norme{Z_q}^2_{L_{\bar A_0}^2}+\kappa_0\norme{Z_{2,q}}^2_{L^2}\\\lesssim
c_q2^{-q\frac d2}\bigl(\norme{Z}_{\dot{\mathbb{B}}^{\frac{d}{2}}_{2,1}}\norme{(W,\nabla Z_2)}_{\dot{\mathbb{B}}^{\frac d2}_{2,1}}
+\norme{\nabla Z}_{\dot{\mathbb{B}}^{\frac d2}_{2,1}}\norme{Z_2}_{\dot{\mathbb{B}}^{\frac d2}_{2,1}}+ \norme{Z_2}^2_{\dot{\mathbb{B}}^{\frac d2}_{2,1}}\bigr)\norme{Z_q}_{L^2}.\label{eq:L2estimateBF}\end{multline}
\end{Prop}
\begin{proof}
The starting point is still \eqref{eq:Zq1} but we  now take advantage of 
Lemma \ref{dtZbis} and  refine the estimates for $R_q^1$ and \eqref{eq:nablaZ}.
More precisely, we  have 
$$
R_q^1=\sum_{j=1}^d\begin{pmatrix} [\wt A^j_{1,1}(V),\ddq]\d_jZ_1 +
 [\wt A^j_{1,2}(V),\ddq]\d_jZ_2\\[1ex]
  [\wt A^j_{2,1}(V),\ddq]\d_jZ_1+
   [\wt A^j_{2,2}(V),\ddq]\d_jZ_2\end{pmatrix}\cdotp
   $$
   Hence, using Inequality \eqref{eq:com1}, we get for all $\sigma\in]-d/2,d/2+1],$ 
   $$\displaylines{
   \|R_q^1\|_{L^2}\lesssim c_q2^{-q\sigma}\Bigl(
   \|\nabla(\wt A^j_{1,1}(V)),\nabla(\wt A^j_{2,1}(V))\|_{\dot\B^{\frac d2}_{2,1}}
   \|Z_1\|_{\dot\B^\sigma_{2,1}}\hfill\cr\hfill + \|\nabla(\wt A^j_{1,2}(V)),\nabla(\wt A^j_{2,2}(V))\|_{\dot\B^{\frac d2}_{2,1}}\|Z_2\|_{\dot\B^\sigma_{2,1}} \Bigr)\cdotp}$$
   At this point, one can use that \eqref{StructAssum} ensures that 
   for all $j\in\{1,\cdots,d\}$ and $k\in\{1,2\},$ there exist 
   a linear map $h$ and a smooth map $F$ such that  
   $\wt A^j_{k,1}(V)-\wt A^j_{k,1}(\bar V)= h(Z_2) F(Z).$ 
   Consequently, product laws and composition estimates give us 
   $$\begin{aligned}
   \|\nabla(\wt A^j_{k,1}(V))\|_{\dot\B^{\frac d2}_{2,1}} &\lesssim \|\nabla(h(Z_2)) \otimes F(Z)\|_{\dot\B^{\frac d2}_{2,1}} + \|h(Z_2) \otimes \nabla(F(Z))\|_{\dot\B^{\frac d2}_{2,1}}\\
   &\lesssim \|\nabla Z_2\|_{\dot\B^{\frac d2}_{2,1}} (1+ \|Z\|_{\dot\B^{\frac d2}_{2,1}})
   +\|Z_2\|_{\dot\B^{\frac d2}_{2,1}}\|\nabla Z\|_{\dot\B^{\frac d2}_{2,1}}\\
   &\lesssim   \|\nabla Z_2\|_{\dot\B^{\frac d2}_{2,1}}  +\|Z_2\|_{\dot\B^{\frac d2}_{2,1}}\|\nabla Z\|_{\dot\B^{\frac d2}_{2,1}},\end{aligned}
   $$
  whence  
   \begin{equation}\label{eq:Rq1}
    \|R_q^1\|_{L^2}\lesssim c_q2^{-q\sigma}
    \Bigl( \|\nabla Z_2\|_{\dot\B^{\frac d2}_{2,1}}  +\|Z_2\|_{\dot\B^{\frac d2}_{2,1}}\|\nabla Z\|_{\dot\B^{\frac d2}_{2,1}}\Bigr) \|Z_1\|_{\dot\B^\sigma_{2,1}}+ \|\nabla Z\|_{\dot\B^{\frac d2}_{2,1}}\|Z_2\|_{\dot\B^\sigma_{2,1}}.
    \end{equation}
    Let us also observe that Inequality \eqref{eq:R''} of Proposition \ref{ProprV} gives us
 $$ \|r(Z)\|_{\dot\B^{\frac d2+1}_{2,1}}\lesssim 
   \|Z_2\|_{\dot\B^{\frac d2+1}_{2,1}}\|Z_2\|_{\dot\B^{\frac d2}_{2,1}}
   +\|Z\|_{\dot\B^{\frac d2+1}_{2,1}} \|Z_2\|_{\dot\B^{\frac d2}_{2,1}}^2.$$
    Remembering that
    $$  \|R_q^2\|_{L^2}\lesssim c_q 2^{-q(\frac d2+1)} \|\nabla Z\|_{\dot\B^{\frac d2}_{2,1}}
    \|\d_tZ\|_{\dot\B^{\frac d2}_{2,1}},$$
    and  using Lemma \ref{dtZbis} as well as \eqref{eq:nablaZ} and \eqref{eq:dta0}, we eventually get \eqref{eq:L2estimateHF}.
\medbreak
For proving \eqref{eq:L2estimateBF}, the starting point is \eqref{eq:Zq2}.
The term corresponding to $R_q^1$ (resp. $r(Z)$) can be bounded according to \eqref{eq:Rq1} 
(resp. \eqref{eq:R''}) with $\sigma=d/2.$
In order to bound the term corresponding to $R_q^3,$ 
we observe that, in light of Lemma \ref{dtZbis},
$$\begin{aligned}
\|(\wt A^0_{1,1}(V)-\wt A^0_{1,1}(\bar V))\d_tZ_1\|_{\dot\B^{\frac d2}_{2,1}}
&\lesssim \|Z\|_{\dot\B^{\frac d2}_{2,1}} \|\d_tZ_1\|_{\dot\B^{\frac d2}_{2,1}}\\
&\lesssim  \|Z\|_{\dot\B^{\frac d2}_{2,1}}
\bigl(\norme{\nabla Z_2}_{\dot{\mathbb{B}}^{\frac d2}_{2,1}}+ \|Z_2\|_{\dot\B^{\frac d2}_{2,1}}
\|\nabla Z_1\|_{\dot\B^{\frac d2}_{2,1}}\bigr),\\
\|(\wt A^0_{2,2}(V)-\wt A^0_{2,2}(\bar V))\d_tZ_2\|_{\dot\B^{\frac d2}_{2,1}}
&\lesssim \|Z\|_{\dot\B^{\frac d2}_{2,1}} \|W\|_{\dot\B^{\frac d2}_{2,1}}.
\end{aligned}
$$
Finally, we have to  refine Inequality \eqref{eq:nablaZ}. To this end, we use the decomposition
$$\displaylines{\int_{\mathbb{R}^d}\sum_{j=1}^d \partial_{j}(\wt A^j(V))Z_q\cdotp Z_q=
\sum_{j=1}^d\int_{\R^d}\biggl(\d_j\bigl(\wt A^j_{1,1}(V)\bigr) Z_{1,q}\cdot Z_{1,q}+ 
\d_j\bigl(\wt A^j_{2,1}(V)\bigr) Z_{2,q}\cdot Z_{1,q}\hfill\cr\hfill+ 
\d_j\bigl(\wt A^j_{1,2}(V)\bigr) Z_{1,q}\cdot Z_{2,q}+ 
\d_j\bigl(\wt A^j_{2,2}(V)\bigr) Z_{2,q}\cdot Z_{2,q}\biggr)\cdotp}$$
The structure assumptions \eqref{StructAssum} and the symmetry of the system
ensure  that 
$$\|\d_j\bigl(\wt A^j_{1,1}(V)\bigr)\|_{L^\infty}+
\|\d_j\bigl(\wt A^j_{1,2}(V)\bigr)\|_{L^\infty}
+\|\d_j\bigl(\wt A^j_{2,1}(V)\bigr)\|_{L^\infty}\lesssim \|Z_2\|_{L^\infty}\|\nabla Z\|_{L^\infty}
+\|\nabla Z_2\|_{L^\infty}.$$
Hence, remembering \eqref{eq:smallZbis}, 
$$\int_{\mathbb{R}^d}\sum_{j=1}^d \partial_{j}(\wt A^j(V))Z_q\cdotp Z_q\lesssim
\bigl( \|Z_2\|_{L^\infty}\|\nabla Z\|_{L^\infty}
+\|\nabla Z_2\|_{L^\infty}\bigr)\|Z_{1,q}\|_{L^2}^2+\|\nabla Z\|_{L^\infty}\|Z_{2,q}\|_{L^2}^2.$$
 Plugging all the above inequalities in  \eqref{eq:Zq2}, we end up with 
 \eqref{eq:L2estimateBF}. 
   \end{proof}

\subsubsection{Cross estimates}

 Remember that for all $q\in\Z,$ we have
 \begin{equation}\label{eq:Iqbis}\frac{d}{dt}\mathcal{I}_q
+\frac{2^q}2 \sum_{k=1}^{n-1}\varepsilon_k \int_{\R^d} |NM_\omega^k \wh Z_q|^2\,d\xi
\leq \frac {2^{-q}\kappa_0}2 \|NZ_q\|_{L^2}^2 +  C\|\ddq G\|_{L^2}\|Z_q\|_{L^2}.\end{equation}
In our new regularity context, we have  to add up  $2^q\cI_q$ if $q<0$ (resp. 
$2^{-q}\cI_q$ if $q\geq0$ ) to $\cL_q,$ 
then to multiply by $2^{q\frac d2}$ (resp. 
$2^{q(\frac d2+1)}$).  This amounts to bounding 
$\|G\|_{\dot\B^{\frac d2+1}_{2,1}}^\ell$ and $\|G\|_{\dot\B^{\frac d2}_{2,1}}^h.$
To this end, we have to refine  the estimates  \eqref{eq:G1}, \eqref{eq:G2}
and \eqref{eq:G3} taking our structure assumption   \eqref{StructAssum}  into account. 

As a first, we see that   \eqref{eq:comparaison} and  Proposition \ref{ProprV} ensure that
\begin{equation}\label{G3}
\|G_3\|_{\dot\B^{\frac d2+1}_{2,1}}^\ell+ \|G_3\|_{\dot\B^{\frac d2}_{2,1}}^h\lesssim
\|G_3\|_{\dot\B^{\frac d2}_{2,1}}\lesssim \|Z_2\|_{\dot\B^{\frac d2}_{2,1}}^2.\end{equation}
Next, we have, thanks to Propositions \ref{LP} and \ref{ProprV}, 
\begin{equation}\label{G2}\|G_2\|_{\dot\B^{\frac d2+1}_{2,1}}^\ell+ \|G_2\|_{\dot\B^{\frac d2}_{2,1}}^h\lesssim
\|G_2\|_{\dot\B^{\frac d2+1}_{2,1}}\lesssim 
\|Z\|_{\dot\B^{\frac d2}_{2,1}}\|Z_2\|_{\dot\B^{\frac d2+1}_{2,1}}+ 
\|Z\|_{\dot\B^{\frac d2+1}_{2,1}}\|Z_2\|_{\dot\B^{\frac d2}_{2,1}}.\end{equation}
In order to improve the estimate for $G_1,$ we use that $\bar A_0^{-1}G_1$ 
is the sum for $j=1$ to $d$ of
$$
\begin{pmatrix}
\bigl((\wt A^0_{1,1}(V))^{-1}\wt A^j_{1,1}(V)-(\bar A^0_{1,1})^{-1}\bar A^j_{1,1}\bigr)\d_jZ_1
+\bigl((\wt A^0_{1,1}(V))^{-1}\wt A^j_{1,2}(V)-(\bar A^0_{1,1})^{-1}\bar A^j_{1,2}\bigr)\d_jZ_2\\
\bigl((\wt A^0_{2,2}(V))^{-1}\wt A^j_{2,1}(V)-(\bar A^0_{2,2})^{-1}\bar A^j_{2,1}\bigr)\d_jZ_1
+\bigl((\wt A^0_{2,2}(V))^{-1}\wt A^j_{2,2}(V)-(\bar A^0_{2,2})^{-1}\bar A^j_{2,2}\bigr)\d_jZ_2
\end{pmatrix}\cdotp
$$
Hence, owing to  \eqref{StructAssum},  we just have 
$$
\|G_1\|_{\dot\B^{\frac d2}_{2,1}} \lesssim \|Z_2\|_{\dot\B^{\frac d2}_{2,1}}\|\nabla Z_1\|_{\dot\B^{\frac d2}_{2,1}}+ \|Z\|_{\dot\B^{\frac d2}_{2,1}}\|\nabla Z_2\|_{\dot\B^{\frac d2}_{2,1}}.$$
Together with \eqref{G3} and \eqref{G2}, we can conclude that 
\begin{equation}\label{eq:Gq}
\|G\|_{\dot\B^{\frac d2+1}_{2,1}}^\ell+ \|G\|_{\dot\B^{\frac d2}_{2,1}}^h
\lesssim \|Z_2\|_{\dot\B^{\frac d2}_{2,1}}^2+ \|Z_2\|_{\dot\B^{\frac d2}_{2,1}}\|Z\|_{\dot\B^{\frac d2+1}_{2,1}}+ \|Z\|_{\dot\B^{\frac d2}_{2,1}}\|Z_2\|_{\dot\B^{\frac d2+1}_{2,1}}.\end{equation}

\subsubsection{Provisional assessment} 

Let $\cL'\triangleq \sum_{q<0} 2^{q\frac d2}\sqrt\cL_q+ \sum_{q\geq0} 2^{q(\frac d2+1)}\sqrt\cL_q.$
Putting together Inequalities \eqref{eq:L2estimateHF}, \eqref{eq:L2estimateBF}, \eqref{eq:Iqbis}
and \eqref{eq:Gq},  using Lemma \ref{SimpliCarre} and discarding the redundant terms, we end up with 
\begin{multline}\label{eq:cL'}
\cL'(t)+\kappa_0\int_0^t\bigl(\|Z\|_{\dot\B^{\frac d2+2}_{2,1}}^\ell+ \|Z\|_{\dot\B^{\frac d2+1}_{2,1}}^h\bigr)
\leq \cL'(0)\\ +C\int_0^t \bigl(\|(W,\nabla Z_2)\|_{\dot\B^{\frac d2}_{2,1}} +
\|Z_2\|_{\dot\B^{\frac d2}_{2,1}}\|Z\|_{\dot\B^{\frac d2+1}_{2,1}}\bigr)\cL'
+C\int_0^t\bigl(\|Z_2\|_{\dot\B^{\frac d2}_{2,1}}^2+ \|Z_2\|_{\dot\B^{\frac d2}_{2,1}}\|Z\|_{\dot\B^{\frac d2+1}_{2,1}}+\|Z\|_{\dot\B^{\frac d2+1}_{2,1}}^2\bigr)\cdotp\end{multline}

In order to close the estimates, we need to exhibit the $L^1$-in-time integrability of $W$ and  $\nabla Z_2$ in $\dot{\mathbb{B}}^{\frac{d}{2}}_{2,1}$ and the $L^2$-in-time integrability of $Z_2$ in $\dot{\mathbb{B}}^{\frac{d}{2}}_{2,1}$. 

\subsubsection{Bounds for the damped mode} 

With the notations we used to prove  \eqref{eq:W0},  remember that
\begin{equation}\label{eq:W0bis}\frac12\frac d{dt}\|W_q\|_{L^2_{\bar A^0_{2,2}}}^2 +\kappa_0\|W_q\|_{L^2}^2\leq  
\bigl(\|\ddq h_1\|_{L^2} + C\|\ddq h_2\|_{L^2}+ C\|\ddq h_3\|_{L^2}\bigr)\|W_q\|_{L^2}.
\end{equation}
{}From Lemma \ref{dtZbis}, Propositions \ref{LP}, \ref{Composition} and \ref{ProprV} and, since 
$$
\d_tQ(Z)=D_{Z_1}Q(Z)\d_tZ_1+ D_{Z_2}Q(Z)\d_tZ_2,$$  we readily get
$$\displaylines{\|h_1\|_{\dot\B^{\frac d2}_{2,1}}^\ell + \|h_1\|_{\dot\B^{\frac d2+1}_{2,1}}^\ell
\lesssim  \|h_1\|_{\dot\B^{\frac d2}_{2,1}}
\lesssim \|Z\|_{\dot\B^{\frac d2}_{2,1}} \|W\|_{\dot\B^{\frac d2}_{2,1}},\cr
\|h_3\|_{\dot\B^{\frac d2}_{2,1}}^\ell + \|h_3\|_{\dot\B^{\frac d2+1}_{2,1}}^\ell
\lesssim  \|h_3\|_{\dot\B^{\frac d2}_{2,1}}
\lesssim \|Z_2\|_{\dot\B^{\frac d2}_{2,1}}^2\|\nabla Z\|_{\dot\B^{\frac d2}_{2,1}}
+ \|Z_2\|_{\dot\B^{\frac d2}_{2,1}} \|W\|_{\dot\B^{\frac d2}_{2,1}}.}$$

 For bounding  $h_2,$  we need to refine the decomposition we did in the previous 
 section. More precisely, we now write that  for all $j\in\{1,\cdots,d\},$ 
 $$\displaylines{\d_t(A^j_{2,1}(V)\d_j Z_1 +A^j_{2,2}(V)\d_j Z_2) = D_{V_1}A^j_{2,1}(V)\d_t Z_1\d_jZ_1 
 +D_{V_2}A^j_{2,1}(V)\d_t Z_2\d_jZ_1\hfill\cr\hfill 
 +A^j_{2,1}(\bar V)\d_t\d_jZ_1  +  \bigl(A^j_{2,1}(V)-A^j_{2,1}(\bar V)\bigr)\d_t\d_jZ_1
\hfill\cr\hfill+ D_VA^j_{2,2}(V)\d_t Z\d_jZ_2 
 +A^j_{2,2}(\bar V)\d_t\d_jZ_2+ \bigl(A^j_{2,2}(V)-A^j_{2,2}(\bar V)\bigr)\d_t\d_jZ_2.}$$
Since   $A^j_{2,1}$ is linear with respect to $V_2,$ we  get
after using  \eqref{eq:dtZ1} and Lemma \ref{dtZbis} that  
$$\displaylines{
\|\ddq h_2\|_{L^2}\lesssim \|\nabla^2Z_{2,q}\|_{L^2} +\|\nabla W_q\|_{L^2}
+c_q2^{-q\frac d2}\Bigl(\|\nabla Z\|_{\dot\B^{\frac d2}_{2,1}}^2+ \|Z_2\|_{\dot\B^{\frac d2}_{2,1}}^2
\hfill\cr\hfill+\|W\|_{\dot\B^{\frac d2}_{2,1}}\|(Z,\nabla Z)\|_{\dot\B^{\frac d2}_{2,1}} +\|Z_2\|_{\dot\B^{\frac d2}_{2,1}}
\|\nabla Z\|_{\dot\B^{\frac d2}_{2,1}} +\|Z\|_{\dot\B^{\frac d2}_{2,1}}
\|\nabla Z_2\|_{\dot\B^{\frac d2}_{2,1}}\Bigr)\cdotp}$$
Hence, reverting to \eqref{eq:W0bis}, using Lemma \ref{SimpliCarre} and keeping 
the  notation  \eqref{eq:Ws}, we end up for $\sigma\in[d/2,d/2+1]$  with 
\begin{multline}\label{eq:Wd2b}
\cW^{\sigma}(t) +  \kappa_0\int_0^t 
\|W\|_{\dot{\mathbb{B}}^{\sigma}_{2,1}}^\ell\leq \cW^{\sigma}(0) 
+C\int_0^t\|(\nabla^2 Z_2,\nabla W)\|^\ell_{\dot\B^{\sigma}_{2,1}} 
\\+C\int_0^t \Bigl( \|Z_2\|_{\dot\B^{\frac d2}_{2,1}}^2+\|\nabla Z\|_{\dot\B^{\frac d2}_{2,1}}^2
+\|W\|_{\dot\B^{\frac d2}_{2,1}}\|(Z,\nabla Z)\|_{\dot\B^{\frac d2}_{2,1}} +\|Z_2\|_{\dot\B^{\frac d2}_{2,1}}
\|\nabla Z\|_{\dot\B^{\frac d2}_{2,1}} +\|Z\|_{\dot\B^{\frac d2}_{2,1}}
\|\nabla Z_2\|_{\dot\B^{\frac d2}_{2,1}}\Bigr)\cdotp
\end{multline}
In order to compare $W$ with $Z_2,$ one can use the decomposition:
$$\displaylines{W-Z_2=\mathds{L}^{-1}\sum_{j=1}^d\Bigl(\bar A^j_{2,1}\d_jZ_1+\bigl(\wt A^j_{2,1}(V)-\bar A^j_{2,1}\bigr)\d_jZ_1
+\bar A^j_{2,2}\d_jZ_2+\bigl(A^j_{2,2}(V)-\bar A^j_{2,2}\bigr)\d_jZ_2\Bigr)\hfill\cr\hfill-\mathds{L}^{-1}Q(Z),}$$
which implies that
\begin{equation}\label{eq:WZ2h}
\|W- Z_2\|_{\dot\B^{s}_{2,1}}^h
\lesssim \|\nabla  Z\|_{\dot\B^{\frac d2}_{2,1}}^h 
+ \|Z_2\|_{\dot\B^{\frac d2}_{2,1}}\|\nabla Z_1\|_{\dot\B^{\frac d2}_{2,1}}
+ \|Z\|_{\dot\B^{\frac d2}_{2,1}}\|\nabla Z_2\|_{\dot\B^{\frac d2}_{2,1}}+\|Z_2\|_{\dot\B^{\frac d2}_{2,1}}^2\end{equation}
and that,  for all $s\geq d/2,$ 
\begin{equation}\label{eq:WZ2l}
\|W- Z_2\|_{\dot\B^{s}_{2,1}}^\ell
\lesssim \|\nabla  Z\|_{\dot\B^{s}_{2,1}}^\ell 
+ \|Z_2\|_{\dot\B^{\frac d2}_{2,1}}\|\nabla Z_1\|_{\dot\B^{\frac d2}_{2,1}}
+ \|Z\|_{\dot\B^{\frac d2}_{2,1}}\|\nabla Z_2\|_{\dot\B^{\frac d2}_{2,1}}+\|Z_2\|_{\dot\B^{\frac d2}_{2,1}}^2.\end{equation}

\subsubsection{Closure of the estimates}

As in the previous section, if we set
$$\wt\cL'\triangleq \cL'+\e \cW^{\frac d2+1}+\e'\cW^{\frac d2}\andf 
\wt\cH'\triangleq  \|Z\|_{\dot\B^{\frac d2+2}_{2,1}}^\ell+ \|Z\|_{\dot\B^{\frac d2+1}_{2,1}}^h+\e\|W\|_{\dot{\mathbb{B}}^{\frac d2+1}_{2,1}}^\ell+\e'\|W\|_{\dot{\mathbb{B}}^{\frac d2}_{2,1}}^\ell$$
with suitable $\e$ and $\e',$ then putting together \eqref{eq:cL'} and \eqref{eq:Wd2b} yields
\begin{multline}\label{eq:wtcL'}
\wt\cL'(t) +\kappa_0\int_0^t\wt\cH' \leq\wt\cL'(0)+ C\int_0^t \bigl(\|Z_2\|_{\dot\B^{\frac d2+1}_{2,1}} +
\|W\|_{\dot\B^{\frac d2}_{2,1}} +
\|Z_2\|_{\dot\B^{\frac d2}_{2,1}}\|Z\|_{\dot\B^{\frac d2+1}_{2,1}}\bigr)\cL'
\\+C\int_0^t\bigl(\|Z_2\|_{\dot\B^{\frac d2}_{2,1}}^2+ \|Z_2\|_{\dot\B^{\frac d2}_{2,1}}\|Z\|_{\dot\B^{\frac d2+1}_{2,1}}+\|Z\|_{\dot\B^{\frac d2+1}_{2,1}}^2\bigr)\cdotp\end{multline}
Note that we have
$$
\bigl(\|Z_2\|_{\dot\B^{\frac d2}_{2,1}}^h\bigr)^2\lesssim
\bigl(\|Z_2\|_{\dot\B^{\frac d2+1}_{2,1}}^h\bigr)^2\lesssim \|Z\|_{\dot\B^{\frac d2+1}_{2,1}}^h\cL'
$$
and 
\begin{eqnarray}\label{eq:above}
\|Z\|_{\dot\B^{\frac d2+1}_{2,1}}^2+
\|Z_2\|_{\dot\B^{\frac d2}_{2,1}}^h\|Z\|_{\dot\B^{\frac d2+1}_{2,1}}&\lesssim&
\|Z\|_{\dot\B^{\frac d2+1}_{2,1}}^2\nonumber\\
&\lesssim& \|Z\|_{\dot\B^{\frac d2+2}_{2,1}}^\ell  \|Z\|_{\dot\B^{\frac d2}_{2,1}}^\ell
+\bigl(\|Z\|_{\dot\B^{\frac d2+1}_{2,1}}^h\bigr)^2\nonumber\\
&\lesssim& \bigl(\|Z\|_{\dot\B^{\frac d2+2}_{2,1}}^\ell+ \|Z\|_{\dot\B^{\frac d2+1}_{2,1}}^h\bigr)\cL'.
\end{eqnarray}
Hence, using also \eqref{eq:WZ2h}, and \eqref{eq:WZ2l}  with $s=d/2+1,$ we see that \eqref{eq:wtcL'} becomes just
$$
\wt\cL'(t) +\kappa_0\int_0^t\wt\cH' \\\leq\wt\cL'(0)+ C\int_0^t \wt\cH' \cL'
+\int_0^t\|Z_2\|_{\dot\B^{\frac d2}_{2,1}}^\ell\bigl(\|Z_2\|_{\dot\B^{\frac d2}_{2,1}}^\ell+\|Z\|_{\dot\B^{\frac d2+1}_{2,1}}\bigr)\cdotp$$
To handle the last integral, let us write that, by virtue of \eqref{eq:WZ2l} with $s=d/2,$ we have 
$$
(\|Z_2\|_{\dot\B^{\frac d2}_{2,1}}^\ell)^2\lesssim \bigl(\|W\|^\ell_{\dot\B^{\frac d2}_{2,1}}\bigr)^2+\|Z_2\|_{\dot\B^{\frac d2}_{2,1}}^4
+\bigl(\|Z\|^\ell_{\dot\B^{\frac d2+1}_{2,1}}\bigr)^2
+ \|Z_2\|_{\dot\B^{\frac d2}_{2,1}}^2\|\nabla Z_1\|_{\dot\B^{\frac d2}_{2,1}}^2
+ \|Z\|_{\dot\B^{\frac d2}_{2,1}}^2\|\nabla Z_2\|_{\dot\B^{\frac d2}_{2,1}}^2.
$$
The last three terms of the right-hand side may be bounded (owing to \eqref{eq:above} and to  \eqref{eq:smallZ})  by $\wt\cH'\cL',$ and we have
$$
\bigl(\|W\|^\ell_{\dot\B^{\frac d2}_{2,1}}\bigr)^2\lesssim \wt\cH'\wt\cL'$$
$$\|Z_2\|_{\dot\B^{\frac d2}_{2,1}}^4\lesssim \|Z\|_{\dot\B^{\frac d2}_{2,1}}^2\|Z_2\|_{\dot\B^{\frac d2}_{2,1}}^2\lesssim (\mathcal{L}')^2\|Z_2\|_{\dot\B^{\frac d2}_{2,1}}^2. $$ 
Finally, we have 
$$\|Z_2\|_{\dot\B^{\frac d2}_{2,1}}^\ell\|Z\|_{\dot\B^{\frac d2+1}_{2,1}}
\lesssim \|W\|_{\dot\B^{\frac d2}_{2,1}}^\ell\|Z\|_{\dot\B^{\frac d2+1}_{2,1}}+
\|Z\|_{\dot\B^{\frac d2+1}_{2,1}}^2 +\|Z_2\|_{\dot\B^{\frac d2}_{2,1}}^2\|Z\|_{\dot\B^{\frac d2+1}_{2,1}}.$$
 Hence there exists a constant $C$ (that may depend on $\e$ and $\e'$
 but not on the solution) such that for all $t\in[0,T],$ we have
 $$\wt\cL'(t) +\kappa_0\int_0^t\wt\cH' \\\leq\wt\cL'(0)+ C\int_0^t \wt\cH' \wt\cL'+ C\int_0^t (\cL'+ (\cL')^2)\|Z_2\|_{\dot\B^{\frac d2}_{2,1}}^2.$$
  Then, one can conclude exactly has in the previous section that  if
 $\wt\cL'(0)$ (or, equivalently, $\cZ'_0$) is small enough, then 
  $\wt\cL'$ is 
 a Lyapunov functional  such   that  for some  (new) positive real numbers $\kappa_0$ and $C,$
 \begin{equation}\label{eq:cL'afsd} \wt\cL'(t) +\kappa_0\int_0^t\wt\cH'\leq\wt\cL'(0)+C\wt\cL'(0)\int_0^t\|Z_2\|^2_{\dot\B^{\frac d2}_{2,1}}.\end{equation}
Furthermore, \eqref{eq:WZ2l} with $s=d/2$ ensures that
$$\displaylines{\quad\|Z_2\|_{L^2_T(\dot\B^{\frac d2}_{2,1})}^\ell \lesssim \|W\|_{L^2_T(\dot\B^{\frac d2}_{2,1})}^\ell 
+  \|\nabla Z\|_{L^2_T(\dot\B^{\frac d2}_{2,1})}^\ell +
\|Z_2\|_{L_T^\infty(\dot\B^{\frac d2}_{2,1})} \|\nabla Z\|_{L_T^2(\dot\B^{\frac d2}_{2,1})} 
\hfill\cr\hfill+\|Z\|_{L_T^\infty(\dot\B^{\frac d2}_{2,1})}\|\nabla Z_2\|_{L^2_T(\dot\B^{\frac d2}_{2,1})}+ \|Z_2\|_{L^\infty_T(\dot\B^{\frac d2}_{2,1})}\|Z_2\|^h_{L^2_T(\dot\B^{\frac d2}_{2,1})}+ \|Z_2\|_{L^\infty_T(\dot\B^{\frac d2}_{2,1})}\|Z_2\|^\ell_{L^2_T(\dot\B^{\frac d2}_{2,1})}.\quad}$$
Inequality \eqref{eq:cL'afsd} combined with \eqref{eq:smallZ} and an obvious interpolation inequality thus yields
\begin{equation}\label{eq:Z_2}
\|Z_2\|_{L^2_T(\dot\B^{\frac d2}_{2,1})}^\ell \lesssim \cZ'(0).\end{equation}
Similarly, using again \eqref{eq:WZ2l} but  with $s=d/2+1,$ we see that
$$\displaylines{\quad\|Z_2\|_{L^1_T(\dot\B^{\frac d2+1}_{2,1})}^\ell \lesssim \|W\|_{L^1_T(\dot\B^{\frac d2+1}_{2,1})}^\ell 
+  \|\nabla Z\|_{L^1_T(\dot\B^{\frac d2+1}_{2,1})}^\ell +
\|Z_2\|_{L_T^2(\dot\B^{\frac d2}_{2,1})} \|\nabla Z\|_{L_T^2(\dot\B^{\frac d2}_{2,1})} 
\hfill\cr\hfill+\|Z\|_{L_T^\infty(\dot\B^{\frac d2}_{2,1})}\|Z_2\|_{L^1_T(\dot\B^{\frac d2+1}_{2,1})}^h
+\|Z\|_{L_T^\infty(\dot\B^{\frac d2}_{2,1})}\|Z_2\|_{L^1_T(\dot\B^{\frac d2+1}_{2,1})}^\ell.\quad}$$
In light of \eqref{eq:smallZ}, the last term may be absorbed by the left-hand side 
and all the other terms may be bounded either through \eqref{eq:cL'afsd} or
through \eqref{eq:Z_2}.

{}From this point, the rest of the proof of this theorem essentially follows the lines 
of the previous section. \qed


\subsection{The isentropic compressible Euler System with damping}

We consider 
\begin{equation} \left\{ \begin{matrix}\partial_t\rho +\text{div}(\rho u)=0,\\[1ex] \partial_t(\rho u)+\text{div}(\rho u\otimes u)+\nabla P+\lambda\rho u=0, \end{matrix} \right.
\label{CED1}\end{equation} 
with $\lambda>0$ and where $P$ is a   (smooth) pressure law  satisfying\footnote{For simplicity we assume
that the reference density is $1$ so that the steady state is $\bar V=(1,0).$}
\begin{equation}\label{Pression1}P'(\rho)>0\ \text{ for }\ \rho\ \text{ close to }\  1 \andf P'(1)=1.\end{equation}

Considering the new unknown $\displaystyle n(\rho)=\int_1^\rho \frac{P'(s)}{s}\,ds$,
we can rewrite \eqref{CED1}   under the form
\begin{equation} \left\{ \begin{aligned} &\partial_tn+u\cdot\nabla n+\textrm{div}\,u+G(n)\textrm{div}\,u=0,\\ 
&\partial_tu+u\cdot\nabla u+\nabla n+\lambda u=0, \end{aligned} \right.\label{CED4}
\end{equation} where $G(n)$ is defined by the relation\footnote{Observe
that $\rho\mapsto n(\rho)$ is a smooth diffeomorphism  from a neighborhood
of $1$ to a neighborhood of $0$.}
$G(n(\rho))=P'(\rho)-1.$
\medbreak
In order to state our global existence for \eqref{CED4},  we need to introduce the following notations:
 $$\displaylines{z^{\ell,\lambda}\triangleq \sum_{2^q\leq \lambda}\ddq z,\qquad  z^{h,\lambda}\triangleq\sum_{2^q>\lambda}\ddq z,\cr 
 {\norme{z}^{\ell,\lambda}_{\dot{\mathbb{B}}^{s}_{2,1}}\triangleq \sum_{2^q\leq \lambda }2^{qs}\|\ddq z\|_{L^2} \andf
\norme{z}^{h,\lambda}_{\dot{\mathbb{B}}^{s}_{2,1}}\triangleq \sum_{2^q> \lambda}2^{qs}\|\ddq z\|_{L^2}}.}$$ 
\begin{Thm}\label{ThmEulerd2}  Let $(n_0,u_0)$ be in $\dot{\mathbb{B}}^{\frac{d}{2}}_{2,1}\cap\dot{\mathbb{B}}^{\frac{d}{2}+1}_{2,1}.$ Then, there  exist
two  positive constants $c$ and $C$ depending only on $G$ and on $d,$ such that if 
 \begin{eqnarray*} \norme{(n_0,u_0)}^{\ell,\lambda}_{\dot{\mathbb{B}}^{\frac{d}{2}}_{2,1}} +\lambda^{-1} \norme{(n_0,u_0)}^{h,\lambda}_{\dot{\mathbb{B}}^{\frac{d}{2}+1}_{2,1}} \leq c,
\end{eqnarray*} then System \eqref{CED4}  supplemented with initial data
$(n_0,u_0)$ 
admits a unique global-in-time solution $(n,u)$ in the space  defined by \begin{eqnarray*}
&& (n,u)\in \mathcal{C}_b(\mathbb{R}^+;\dot{\mathbb{B}}^{\frac{d}{2}}_{2,1}\cap\dot{\mathbb{B}}^{\frac{d}{2}+1}_{2,1}),\;\;\; 
(n^{h,\lambda},u^{h,\lambda})\in L^1(\mathbb{R}^+;\dot{\mathbb{B}}^{\frac{d}{2}+1}_{2,1}), \,\;\;\; n^{\ell,\lambda}\in L^1(\mathbb{R}^+,\dot{\mathbb{B}}^{\frac{d}{2}+2}_{2,1}), \\&& u^{\ell,\lambda}\in L^1(\mathbb{R}^+;\dot{\mathbb{B}}^{\frac{d}{2}+1}_{2,1}),  \;\;\; u\in L^2(\mathbb{R}^+;\dot{\mathbb{B}}^{\frac{d}{2}}_{2,1}) \andf \nabla n+\lambda u\in L^1(\mathbb{R}^+;\dot{\mathbb{B}}^{\frac{d}{2}}_{2,1}).
\end{eqnarray*}
Moreover we have the following a priori estimate:
\begin{equation}\label{eq:Zlambda}
\cZ_\lambda (t)\lesssim \norme{(n_0,u_0)}^{\ell,\lambda}_{\dot{\mathbb{B}}^{\frac{d}{2}}_{2,1}}+\lambda^{-1}\norme{(n_0,u_0)}^{h,\lambda}_{\dot{\mathbb{B}}^{\frac{d}{2}+1}_{2,1}} \quad\hbox{for all } t\geq 0\end{equation}
where
$$\displaylines{
\cZ_\lambda(T)\triangleq\norme{(n,u)}^{\ell,\lambda}_{L^\infty_T(\dot{\mathbb{B}}^{\frac{d}{2}}_{2,1})}
+\lambda^{-1}\norme{(n,u)}^{h,\lambda}_{L_T^\infty(\dot{\mathbb{B}}^{\frac{d}{2}+1}_{2,1})}\hfill\cr\hfill
+\lambda^{-1}\norme{n}^{\ell,\lambda}_{L^1_T(\dot{\mathbb{B}}^{\frac{d}{2}+2}_{2,1})}+\norme{(n,u)}^{h,\lambda}_{L^1_T(\dot{\mathbb{B}}^{\frac{d}{2}+1}_{2,1})}
+\norme{u}^{\ell,\lambda}_{L^1_T(\dot{\mathbb{B}}^{\frac{d}{2}+1}_{2,1})}+\lambda^{1/2}\norme{u}^{\ell,\lambda}_{L^2_T(\dot{\mathbb{B}}^{\frac{d}{2}}_{2,1})}+\norme{\nabla n+\lambda u}^{\ell,\lambda}_{L^1_T(\dot{\mathbb{B}}^{\frac{d}{2}}_{2,1})}.}$$
If furthermore, $(n_0,u_0)$ belongs to $\dot\B^{-\sigma_1}_{2,\infty}$ for some $\sigma_1\in]-d/2,d/2],$
then  the solution $(n,u)$ satisfies \eqref{eq:Zs1} and the decay  estimates mentioned at the end
of Theorem \ref{Thmd2} hold true. \end{Thm}
\begin{proof}
Performing the rescaling 
$$(n,u)(t,x)\triangleq (\wt n, \wt u)(\lambda t,\lambda x)$$ 
reduces the proof to $\lambda=1$ (and  the inverse scaling
will eventually give the desired dependency with respect to $\lambda$ in the 
above statement). Then, the whole result is a corollary of  Theorem \ref{Thmd2}
provided System \eqref{CED4} satisfies the structural assumption \eqref{StructAssum} at $0.$ 
Indeed, one can take as a symmetrizer the matrix
$\begin{pmatrix} (1+G(n))^{-1}&0\\0& I_d\end{pmatrix}$ 
where the first diagonal block is of size $1\times 1$ and the second one, of size  $d\times d$. 
The blocks of type $A^j_{1,1}$ and $A^j_{2,1}$ depend only (and linearly) on $u,$ 
which is indeed the damped component.  
Finally, the damped mode (in the case $\lambda=1$) is $W=u+\nabla n+u\cdot\nabla u.$
Now,  by virtue of \eqref{eq:Zlambda}, 
$$
\|W-(u+\nabla n)\|_{L^1_T(\dot\B^{\frac d2}_{2,1})} \lesssim 
\|u\|_{L^\infty_T(\dot\B^{\frac d2}_{2,1})} \|\nabla u\|_{L^1_T(\dot\B^{\frac d2}_{2,1})}
\lesssim \bigl(\norme{(n_0,u_0)}^\ell_{\dot{\mathbb{B}}^{\frac{d}{2}}_{2,1}}+\norme{(n_0,u_0)}^h_{\dot{\mathbb{B}}^{\frac{d}{2}+1}_{2,1}}\bigr)^2,$$
 hence $u+\nabla n$ satisfies the same estimates as $W,$ which completes the proof. 
\end{proof}

\section{Appendix}

Here we gather a few technical results that have been used repeatedly in the paper. 
\medbreak
The first one is the justification that one may choose (arbitrarily small) positive parameters $\e_1,\dotsm,\e_{n-1}$
so that, whenever $\wh Z$ satisfies \eqref{eq:Zomega}, Inequality \eqref{eq:I} holds true. 
The proof just consists in bounding suitably the terms of the right-hand side of \eqref{eq:14}. 
\begin{itemize}
\item Terms $\cI^1_k\triangleq \bigl(NM^{k-1}_\omega N\wh Z\cdotp N M_\omega^k\wh Z\bigr)$
with $k\in\{1,\cdots, n-1\}.$ 

Since matrices $M_\omega$ are bounded on $\S^{d-1},$ we may write
$$
\e_k |\cI^1_k|\lesssim \e_k|N\wh Z| |N M_\omega^k\wh Z| \leq \frac{|N\wh Z|^2}{4n\rho}
+C\rho\e_k^2|NM_\omega^k \wh Z|^2.
$$
\item Terms   $\e_k\bigl(NM^{k-1}_\omega \wh Z\cdotp N M_\omega^k N\wh Z\bigr)$ with  $k\in\{2,\cdots, n-1\}$
may be bounded similarly. 
\smallbreak
\item  We have  $\e_1|\bigl(N\wh Z\cdotp N M_\omega N\wh Z\bigr)|
\leq C\e_1|N\wh Z|^2.$
\smallbreak
\item Terms $\cI^2_k\triangleq \rho\bigl(NM_\omega^{k-1}\wh Z\cdotp NM_\omega^{k+1}\wh Z\bigr)$ with $k\in\{1,\cdots, n-2\}.$
We have
$$\begin{aligned}
\e_k |\cI^2_k|&\leq  \e_{k-1}\rho|NM_\omega^{k-1}\wh Z|\, |NM_\omega^{k+1}\wh Z|\\
&\leq \frac14  \rho|N M_\omega^{k-1}\wh Z|^2 +C\rho\frac{\e_k^2}{\e_{k-1}}  |NM_\omega^{k+1}\wh Z|^2. 
\end{aligned}
$$
As we want the two terms to be absorbed by the left-hand side of \eqref{eq:14}, we take $\e_k$ so that 
\begin{equation}\label{eq:e1}
4\e_k^2\leq \e_{k-1}\e_{k+1}.\end{equation}
We keep in mind that $\e_0$ has been set to $(2\pi)^{-d}\kappa_0$ (but   can be taken smaller if needed). 
\smallbreak
\item  Term $\cI^2_{n-1}\triangleq \e_{n-1} \rho\bigl(NM_\omega^{n-2}\wh Z\cdotp NM_\omega^{n}\wh Z\bigr).$
We start with the observation that,  owing to Cayley-Hamilton theorem, there exist coefficients
$c^j_\omega$ (that are uniformly bounded on $\S^{d-1},$ such that
$$M_\omega^n =\sum_{j=0}^{n-1}  c^j_\omega M_\omega^j.$$
Consequently, one may write
$$\begin{aligned}|\cI^2_{n-1}|&\lesssim \e_{n-1}\rho\sum_{j=0}^{n-1} |NM_\omega^{n-2}\wh Z|
|N M_\omega^j\wh Z| \\
&\leq \frac{C\e_{n-1}^2\rho}{\e_j} |NM_\omega^j\wh Z|^2 + \frac{\e_j\rho}4|NM_\omega^j\wh Z|^2.
\end{aligned} $$
Therefore one needs to assume in addition that 
\begin{equation}\label{eq:e2}
4C\e_{n-1}^2\leq \e_{j}\e_{n-2},\qquad j=0,\cdots,n-1.\end{equation}
\end{itemize}
Clearly, one is done if it is possible to find $\e_1,\cdots,\e_{n-1}$ fulfilling 
\eqref{eq:e1} and \eqref{eq:e2}. One can take for instance $\e_k=\e^{m_k}$ with  $\e$ small enough and 
$m_1,\cdots,m_{n-1}$ satisfying for some $\delta>0$ (that can be taken arbitrarily small): 
$$m_k\geq \frac{m_{k-1}+m_{k+1}}2+\delta\andf
m_{n-1}\geq\frac{m_{k}+m_{n-2}}2+ \delta, \quad k=1,\cdots,n-2.$$

\bigbreak
We often used the following  well known result  (see  e.g. \cite{CBD1} for the proof). 
\begin{Lemme}\label{SimpliCarre}
Let $X : [0,T]\to \mathbb{R}^+$ be a continuous function such that $X^2$ is differentiable. Assume that there exists 
 a constant $B\geq 0$ and  a measurable function $A : [0,T]\to \mathbb{R}^+$ 
such that 
 $$\frac{1}{2}\frac{d}{dt}X^2+BX^2\leq AX\quad\hbox{a.e.  on }\ [0,T].$$ 
 Then, for all $t\in[0,T],$ we have
$$X(t)+B\int_0^tX\leq X_0+\int_0^tA.$$
\end{Lemme}
The following estimates  are proved in  e.g.  \cite[Chap. 2]{HJR}.  
\begin{Prop}\label{C1}    The following inequalities hold true: 
\begin{itemize}
\item  If $-d/2<s\leq {d}/{2}+1$, then
 \begin{equation}\label{eq:com1}
2^{qs}\norme{[w,\dot{\Delta}_q]\nabla v}_{L^2}\leq Cc_q\norme{\nabla w}_{\dot{\mathbb{B}}^{\frac{d}{2}}_{2,1}}\norme{v}_{\dot{\mathbb{B}}^{s}_{2,1}}  \with\sum_{q\in\mathbb{Z}}c_q=1.
\end{equation}
\item   If $-d/2\leq s<{d}/{2}+1$, then
\begin{equation}\label{eq:com3} 
\sup_{q\in\Z} 2^{qs}\|[w,\ddq]\nabla v\|_{L^2}\leq C\|\nabla w\|_{\dot\B^{\frac d2}_{2,1}} \|v\|_{\dot\B^{s}_{2,\infty}}.\end{equation}
  \end{itemize}
 \end{Prop}
The following product laws in Besov spaces have been used several times. 
 \begin{Prop} \label{LP} Let $(s,r)$ be in $]0,\infty[\times[1,\infty].$ Then, 
 $\dot{\mathbb{B}}^{s}_{2,r}\cap L^\infty$ is an algebra and we have
\begin{equation}\label{eq:prod1}
\norme{ab}_{\dot{\mathbb{B}}^{s}_{2,r}}\leq C\bigl(\norme{a}_{L^\infty}\norme{b}_{\dot{\mathbb{B}}^{s}_{2,r}}+\norme{a}_{\dot{\mathbb{B}}^{s}_{2,r}}\norme{b}_{L^\infty}\bigr)\cdotp
\end{equation}
If, furthermore, $-d/2<s\leq d/2,$ then the following inequality holds:
\begin{equation}\label{eq:prod2}
\|ab\|_{\dot\B^{s}_{2,1}}\leq C\|a\|_{\dot\B^{\frac d2}_{2,1}}\|b\|_{\dot\B^{s}_{2,1}}.
\end{equation}
Finally,  if  $-d/2<\sigma_1\leq d/2$, then the following inequality holds true: 
\begin{equation}\label{eq:prod3} 
\norme{fg}_{\dot{\mathbb{B}}^{-\sigma_1}_{2,\infty}}\leq C  \norme{f}_{\dot{\mathbb{B}}^{\frac{d}{2}}_{2,1}}\norme{g}_{\dot{\mathbb{B}}^{-\sigma_1}_{2,\infty}}.
\end{equation}

\end{Prop}
The next proposition can be found in \cite{HJR}.
\begin{Prop}\label{Composition}
Let $f$ be a function in $\mathcal{C}^\infty(\mathbb{R})$ such that $f(0)=0$. let $(s_1,s_2)\in]0,\infty[^2$ and \\$(r_1,r_2)\in[1,\infty]^2$. We assume that $s_1<{d}/{2}$ or that $s_1={d}/{2}$ and $r_1=1$.

Then, for every real-valued function $u$ in $\dot{\mathbb{B}}^{s_1}_{2,r_1}\cap\dot{\mathbb{B}}^{s_2}_{2,r_2}\cap L^\infty$, the function $f\circ u$ belongs to $\dot{\mathbb{B}}^{s_1}_{2,r_1}\cap\dot{\mathbb{B}}^{s_2}_{2,r_2}\cap L^\infty$ and we have
$$\norme{f\circ u}_{\dot{\mathbb{B}}^{s_k}_{2,r_k}}\leq C\left(f',\norme{u}_{L^\infty}\right)\norme{u}_{\dot{\mathbb{B}}^{s_k}_{2,r_k}}\quad\hbox{for}\  k\in\{1,2\}.$$
\end{Prop}
As a consequence (see \cite[Cor. 2.66]{HJR}), if $g$ is a $\cC^\infty(\R)$ function such that $g'(0)=0.$  Then,  
for all $u,v$ in $\dot B^s_{2,1}\cap L^\infty$ with $s>0,$ we have
\begin{equation}\label{eq:compo}
\|g(v)-g(u)\|_{\dot B^s_{2,1}} \leq C\Bigl(\|v-u\|_{L^\infty}\|(u,v)\|_{\dot\B^s_{2,1}} + 
\|v-u\|_{\dot\B^s_{2,1}} \|(u,v)\|_{L^\infty}\Bigr)\cdotp
\end{equation}

We used the following result to estimate the remainder  of the dissipative term. 

\begin{Prop} \label{ProprV}
Let $\bar{V}\in\mathcal{M}$ and $Z\triangleq V-\bar{V}.$ 
Define  $r(Z)\triangleq \wt H(\bar V+Z)+LZ,$  $L\triangleq - D_V\wt H(\bar V)$   and $Z_2\triangleq(I_d-\mathcal{P})Z,$ and assume that $r(Z_1,0)=0$ for $Z_1$ in a neighborhood of $0.$  Then, provided 
 $\|Z\|_{\dot\B^{\frac d2}_{2,1}}$ is sufficiently small, the following inequalities hold true:
 \begin{equation}\label{eq:R}
  \|r(Z)\|_{\dot{\mathbb{B}}^\sigma_{2,1}}\lesssim \|Z\|_{\dot{\mathbb{B}}^{\frac d2}_{2,1}}\|Z_2\|_{\dot{\mathbb{B}}^\sigma_{2,1}}
 \quad\hbox{for }\ \sigma\in]-d/2,d/2]\end{equation}
 and, for $\sigma > d/2,$  
 \begin{equation}\label{eq:RR}\norme{r(Z)}_{\dot{\mathbb{B}}^{\sigma}_{2,1}}\lesssim \norme{Z_2}_{\dot{\mathbb{B}}^{\frac{d}{2}}_{2,1}}\norme{Z}_{\dot{\mathbb{B}}^{\sigma}_{2,1}}+\norme{Z_2}_{\dot{\mathbb{B}}^{\sigma}_{2,1}}\norme{Z}_{\dot{\mathbb{B}}^{\frac{d}{2}}_{2,1}}.\end{equation}
 Furthermore, if both $Z^1$ and $Z^2$ are sufficiently small in $\dot\B^{\frac d2}_{2,1}$
 then we have the following estimate for $\wt Z:=Z^1-Z^2$:
 \begin{equation}\label{eq:dr}
  \|r(Z^1)-r(Z^2)\|_{\dot\B^\sigma_{2,1}} \lesssim \|Z^1\|_{\dot\B^{\frac d2}_{2,1}}
 \|\wt Z_2\|_{\dot\B^{\sigma}_{2,1}} + \|\wt Z\|_{\dot\B^{\sigma}_{2,1}}\|Z^2_2\|_{\dot\B^{\frac d2}_{2,1}},
 \qquad   \sigma\in ]0,d/2].\end{equation}
Finally, if $r$ is at least quadratic with respect to $Z_2$ (that is
$D^2_{V_i,V_j} r(0)=0$ for $(i,j)\not=(2,2)$), then we have
\begin{eqnarray}\label{eq:R'} 
  &\|r(Z)\|_{\dot{\mathbb{B}}^\sigma_{2,1}}\lesssim \|Z_2\|_{\dot{\mathbb{B}}^{\frac d2}_{2,1}}\|Z_2\|_{\dot{\mathbb{B}}^\sigma_{2,1}}
& \quad\hbox{for }\ \sigma\in]-d/2,d/2]\\\label{eq:R''} 
 \andf  &\|r(Z)\|_{\dot\B^\sigma_{2,1}}\lesssim 
   \|Z_2\|_{\dot\B^{\sigma}_{2,1}}\|Z_2\|_{\dot\B^{\frac d2}_{2,1}}
   +\|Z\|_{\dot\B^{\sigma}_{2,1}} \|Z_2\|_{\dot\B^{\frac d2}_{2,1}}^2
 & \quad\hbox{for }\ \sigma>d/2.
 \end{eqnarray}

\end{Prop}
\begin{proof}
Since $r(Z_1,0)=0$ for $Z_1$ close to $0,$ the mean value formula gives
$$r(Z_1,Z_2)=\int_0^1 D_{Z_2}r (Z_1,\tau Z_2)\cdotp Z_2\,d\tau.$$
Furthermore, we have $Dr(0)=0$ and thus $D_{Z_2}r(0)=0$. Hence there exists a smooth function $F$ defined near $0$
and such that 
$D_{Z_2}r(Z)= F(Z)\cdotp Z.$
Consequently, there exists a smooth function $G$ vanishing at $0,$ and such that
$$r(Z_1,Z_2)= G(Z)\cdotp Z_2.$$
Granted with the above decomposition, the first two  inequalities readily follow from 
 Propositions \ref{LP} and \ref{Composition}.
 \medbreak 
 To prove \eqref{eq:dr},  we use the decomposition 
 $$r(Z^1)-r(Z^2) = G(Z^1) \cdotp (Z^2_2-Z^1_2) + \bigl(G(Z^2)-G(Z^1)\bigr)\cdotp Z^2_2,$$
 then Propositions \ref{LP} and \ref{Composition}, combined with  
  Corollary 2.66 from \cite{HJR}. 
  \medbreak
  Finally, if $r$ is quadratic with respect to $Z_2$ then there exists a quadratic form 
  $\wt Q$ and a smooth function $F$ such that 
  $r(Z)= \wt Q(Z_2) F(Z),$ whence
  $$  r(Z)= F(0) \wt Q(Z_2) + G(Z)\wt Q(Z_2)\with G(Z)\triangleq F(Z)-F(0).$$
    In the case $\sigma\in]-d/2,d/2],$ we can thus write by virtue of Propositions \ref{LP}
    and \ref{Composition}, 
  $$  \|r(Z)\|_{\dot\B^\sigma_{2,1}}\lesssim  \|\wt Q(Z_2)\|_{\dot\B^{\sigma}_{2,1}} 
  (1+\|Z\|_{\dot\B^{\frac d2}_{2,1}}) \lesssim \|Z_2\|_{\dot\B^{\frac d2}_{2,1}}
   \|Z_2\|_{\dot\B^{\sigma}_{2,1}} $$
  while, if $\sigma >d/2,$ 
  $$ \|r(Z)\|_{\dot\B^\sigma_{2,1}}\lesssim 
   \|\wt Q(Z_2)\|_{\dot\B^{\sigma}_{2,1}}(1  +\|Z\|_{\dot\B^{\frac d2}_{2,1}})
   +\|Z\|_{\dot\B^{\sigma}_{2,1}} \|\wt Q(Z_2)\|_{\dot\B^{\frac d2}_{2,1}},$$
   whence the last inequality.
  \end{proof}
\bibliographystyle{plain}

\bibliography{Biblio3D}

\end{document}